\documentclass[10pt]{article}
\usepackage{amscd}
\usepackage[all]{xy}
\usepackage{bbm}
\usepackage{amssymb}
\usepackage{latexsym}
\usepackage{amsmath}
\usepackage{amsthm}

\newtheorem{Theorem}{Theorem}[section]
\newtheorem{Lemma}[Theorem]{Lemma}

\newtheorem{Proposition}[Theorem]{Proposition}
\newtheorem{Remark}[Theorem]{Remark}

\newtheorem{Definition}[Theorem]{Definition}

\begin{document}
\begin{center}

{\large\bf Elliptic Curves in Moduli Space of Stable Bundles of Rank 3\ }

\vspace{0.2cm}
Liu Min \\
School of Mathematical Sciences, Fudan University, Shanghai 200433, P.R.China\\
E-mail: liumin@amss.ac.cn\\
\end{center}

\begin{abstract}
Let $M$ be the moduli space of rank 3 stable bundles with fixed determinant of degree 1 on a smooth projective curve of genus $g\geq 2$. When $C$ is generic, we show that any essential elliptic curve on $M$ has degree (respect to anti-canonical divisor $-K_{M}$) at least 6, and we give a complete classification for elliptic curves of degree 6, which is not in conformity with Sun's Conjecture. Moreover, if $g>12$, we show that any elliptic curve passing through the generic point of $M$ has degree at least 18.

\end{abstract}

\section{ Introduction }
Let $C$ be a smooth projective curve of genus $g\geq 2$ and $\mathcal{L}$ be a line bundle on $C$ of degree $d$. Let $M:=SU_{C}(r,\mathcal{L})$ be the moduli space of stable vector bundles of rank $r$ and with the fixed determinant $\mathcal{L}$, which is a smooth quasi-projective Fano-variety with $Pic(M)=\mathbb{Z}\cdot\Theta$. And $-K_{M}=2(r,d)\Theta$, where $\Theta$ is an ample divisor(\cite{Ramanan} \cite{Drezet}). Let $B$ be a smooth projective curve of genus $b$. The degree of a curve $\phi: B\rightarrow M$ is defined to be $\text{deg}\phi^{*}(-K_{M})$. It seems quite natural to ask what is the lower bound of degrees and to classify the curves of lower degree.

When $b=0$, i.e., $B=\mathbb{P}^{1}$, it has been proved that any rational curve $\phi: \mathbb{P}^{1}\rightarrow M$ passing through the generic point has degree at least $2r$ provided that $(r, d)=1$. Moreover, it has degree $2r$ if and only if it is a Hecke curve (\cite{Sun05}, Theorem 1). Ramanan \cite{Ramanan} found a family of lines on $M$, i.e., rational curves $\phi: \mathbb{P}^{1}\rightarrow M$ such that $\text{deg}\phi^{*}(-K_{M})=2(r,d)$. And all the lines are determined in \cite{Sun05} and \cite{MokSun09}. In \cite{Liu}, we have studied the small rational curves (i.e., the rational curves have degrees smaller than $2r$ ) on $M$ and estimate the codimension of the locus of the small rational curves when $d=1$; in particular, we determinant all small rational curves when $r=3$ (\cite{Liu}). Thus it is natural to ask what are the situation when $b>0$.

When $b=1$, it may happen that the normalization of $\phi(B)$ is $\mathbb{P}^{1}$. To avoid this case, we only consider the case that $\phi: B\rightarrow M$ is an essential elliptic curve ( cf. \cite{Sun}).
The paper \cite{Sun} is a start to study the case of $b=1$. In \cite{Sun}, Sun constructed essential elliptic curves of degree $6(r,d)$ on $M$, which are called elliptic curves of split type, and essential elliptic curves of degree $6r$ that passing through the generic point of $M$, which are called elliptic curves of Hecke type. Do they exhaust all minimal essential elliptic curves on $M$ ( resp. minimal essential elliptic curves passing through generic point of $M$)? For this, Sun studied the case when $r=2$ and $d=1$, showed that any essential elliptic curve has degree at least 6, it has degree 6 if and only if it is an elliptic curve of split type with minimal degree. Moreover, if $g>4$, Sun showed that any elliptic curve passing through the generic point has degree at least 12. And then Sun conjectured the result holds for any rank $r$ and degree $d$ and gave the following conjecture ( cf. Conjecture 4.8 of \cite{Sun}):

\textbf{Sun's Conjecture:} \emph{Let $\phi: B\rightarrow M=SU_{C}(r,\mathcal{L})$ is an essential elliptic curve defined by a vector bundle $E$ on $C\times B$. Then, when $C$ is a generic curve, we have
$$\text{deg}\phi^{*}(-K_{M})=\Delta(E)\geq 6(r,d)$$
and $\text{deg}\phi^{*}(-K_{M})=6(r,d)$ if and only if it is an elliptic curve of split type with minimal degree. If $\phi: B\rightarrow M$ passes through the generic point and $g>4$, then $\text{deg}\phi^{*}(-K_{M})\geq 6r$.}

 In this paper, we consider the case that $r=3$ and $d=1$, then $M$ is a smooth projective Fano-variety of dimension $8g-8$. When $C$ is generic, we show that any essential elliptic curve $\phi: B\rightarrow M$ has degree at least 6 ( see Theorem \ref{th:4.7}). When $g>12$ and $C$ is generic, we show that any essential elliptic curve $\phi: B\rightarrow M$ passing through the generic point of $M$ have degree at least 18 ( see Theorem \ref{th:4.8}). But an essential elliptic curve of degree 6 may not be an elliptic curve of split type (see Proposition \ref{prop:3.1} and Theorem \ref{th:4.7}).

We give a brief description of the article. In section 2, we recall a degree formula of curves for general case which has proven in \cite{Sun}. In section 3, we recall the the constructions of \textbf{elliptic curves of Hecke type} and \textbf{elliptic curves of split type}. And, we also give a class of degree 6 essential elliptic curves which are not elliptic curves of split type when $r=3$ and $d=1$. In section 4, we prove the main theorems (Theorem \ref{th:4.7} and Theorem \ref{th:4.8}), which partly prove Sun's conjecture for the case $r=3, d=1$. On the other hand, Theorem \ref{th:4.7} also implies that the essential elliptic curves of degree 6 may not be elliptic curves of split type.

\section{ The degree formula of curves in moduli spaces }

Let's recall the degree formula of curves in moduli spaces.

\begin{Lemma}(\cite{Sun})
For any smooth projective curve $B$ of genus $b$, if $\phi: B\rightarrow M$ is defined by a vector bundle $E$ of rank $r$ on $C\times B$. Then
$$\text{deg}\phi^{*}(-K_{M})=c_{2}(\mathcal{E}nd^{0}(E))=2rc_{2}(E)-(r-1)c_{1}(E)^{2}:=\Delta(E).$$
\end{Lemma}

Let $f:X:=C\times B\rightarrow C$ be the projection. Then for any vector bundle $E$ on $X$, there is a relative Harder-Narasimhan filtration (cf. Theorem 2.3.2, page 45 in \cite{HuybrechtsLehn})
$$0=E_{0}\subset E_{1}\subset\cdots\subset E_{n}=E$$
such that $F_{i}=E_{i}/E_{i-1} (i=1,\cdots, n)$ are flat over $C$ and its restriction to general fiber $X_{t}=f^{-1}(t)$ is the Harder-Narasimhan filtration of $E|_{X_{t}}$. Thus $F_{i}$ are semi-stable of slop $\mu_{i}$ at generic fiber of $f: X\rightarrow C$ with $\mu_{1}>\mu_{2}>\cdots >\mu_{n}$. Then we have

\begin{Theorem}(\cite{Sun}) For any vector bundle of rank $r$ on $X$, let
$$0=E_{0}\subset E_{1}\subset\cdots\subset E_{n}=E$$
be the relative Harder-Narasimhan filtration over $C$ with $F_{i}=E_{i}/E_{i-1}$ and $\mu_{i}=\mu(F_{i}|_{f^{-1}(x)})$ for generic $x\in C$. Let $\mu(E)$ and $\mu(E_{i})$ denote the slop of $E|_{\pi^{-1}(b)}$ and $E_{i}|_{\pi^{-1}(b)}$ for generic $b\in B$. Then, if
$$\text{Pic}(C\times B)=\text{Pic}(C)\times \text{Pic}(B),$$
we have the following formula
\begin{equation}
\Delta(E)=2r
  \left( \begin{array}{ll}
              \sum_{i=1}^{n}(c_{2}(F_{i})-\frac{\text{rk}(F_{i})-1}{2\text{rk}(F_{i})}c_{1}(F_{i})^{2})\\
              +\sum_{i=1}^{n-1}(\mu(E)-\mu(E_{i}))\text{rk}(E_{i})(\mu_{i}-\mu_{i+1})\\         \end{array}
  \right).
\label{eq:2.1}
\end{equation}
\label{th:2.2}
\end{Theorem}

\begin{Remark}(i) The assumption $\text{Pic}(C\times B)=\text{Pic}(C)\times \text{Pic}(B)$ is always hold when $B=\mathbb{P}^{1}$;

(ii)The assumption also holds when $B$ is an elliptic curve and $C$ is generic.
\end{Remark}

\begin{Theorem}(\cite{Sun})
 For any torsion free sheaf $\mathcal{F}$ on $X=C\times B$, if its restriction to a fiber of $f: X\rightarrow C$ is semi-stable, then
$$\Delta(\mathcal{F})=2\text{rk}(\mathcal{F})c_{2}(\mathcal{F})-(\text{rk}(\mathcal{F})-1)c_{1}(\mathcal{F})^{2}\geq 0.$$
If the determinants $\{\text{det}(\mathcal{F}^{**})_{x}\}_{x\in C}$ are isomorphic each other, then $\Delta(\mathcal{F})=0$ if and only if $\mathcal{F}$ is locally free and satisfies

$\bullet$ All the bundles $\{\mathcal{F}_{x}:=\mathcal{F}|_{\{x\}\times B}\}_{x\in C}$ are semi-stable and s-equivalent each other.

$\bullet$ All the bundles $\{\mathcal{F}_{y}:=\mathcal{F}|_{C\times \{y\}}\}_{y\in B}$ are isomorphic each other.
 \label{th:2.4}
\end{Theorem}

We will need the following lemma in the later computation, whose proof are straightforward computations.
Recall that $X_{t}=f^{-1}(t)$ denotes the fiber of $f: X\rightarrow C$ and for any vector bundle $\mathcal{F}$ on $X$, $\mathcal{F}_{t}$ denote the restrictions of $\mathcal{F}$ to $X_{t}$.

\begin{Lemma}(\cite{Sun})Let $\mathcal{F}_{t}\rightarrow W\rightarrow 0$ be a locally free quotient and
$$0\rightarrow \mathcal{F}'\rightarrow \mathcal{F}\rightarrow_{X_{t}}W\rightarrow 0$$
be the elementary transformation of $\mathcal{F}$ along $W$ at $X_{t}\subset X$. Then
$$\Delta(\mathcal{F})=\Delta(\mathcal{F}')+2r(\mu(\mathcal{F}_{t})-\mu(W))\text{rk}W.$$
\label{lm2.6}
\end{Lemma}

\section{Examples of Elliptic Curves on Moduli Spaces and Sun's Conjecture}

Recall that given two nonnegative integers $k, l$, a vector bundle $W$ of rank $r$ and degree $d$ on $C$ is $(k,l)$-stable, if, for each proper subbundle $W'$ of $W$, we have
$$\frac{\text{deg}(W')+k}{\text{rk}(W')}<\frac{\text{deg}(W)+k-l}{r}.$$

\begin{Remark}

(i)The usual stability is equivalent to (0,0)-stability.

(ii)If $W$ is $(k,l)$-stable, then $W^{*}$ is $(l,k)$-stable.

(iii)The $(k,l)$-stability is an open condition.
\end{Remark}

Let $M:=SU_{C}(r,\mathcal{L})$ be the moduli space of stable vector bundles of rank $r$ and with the fixed determinant $\mathcal{L}$ and $\text{deg}\mathcal{L}=d$. If $(k,l)$-stable points exist, then the set of $(k,l)$-stable points is open in $M$. And the $(k,l)$-stable points are the so called generic points. So it's natural to ask that does the $(k,l)$-stable point exist?

It's equivalent to estimate the dimension of the subvariety of $M:=SU_{C}(r,\mathcal{L})$ consisting of non-$(k,l)$-stable points. Clearly any such bundle $E$ contains a subbundle $F$ satisfying the inequality
$$\frac{\text{deg}F+k}{\text{rk}F}\geq \frac{\text{deg}E+k-l}{\text{rk}E}=\frac{d+k-l}{r}.$$
By using [proposition 2.6 of \cite{NarasimhanRamanan1}] as in [Lemma 6.7 of  \cite{NarasimhanRamanan1}], we may as well assume that $F$ and $E/F$ are stable and compute the dimension of such bundles $E$. The dimension of a component corresponding to a fixed rank $n$ and degree $\delta$ of $F$ such that $\frac{\delta+k}{n}\geq \frac{d+k-l}{r}$ is majored by $\text{dim}U(n,\delta)+\text{dim}U(r-n,d-\delta)+\text{dim}H^{1}(C,\text{Hom}(E/F, F))-1-g=(r^{2}-1)(g-1)-n(r-n)(g-1)+(nd-r\delta)$. If all the dimensions of the components is strictly smaller than the dimension of M, then $(k,l)$-stable point exists. Thus we hope that $-n(r-n)(g-1)+(nd-r\delta)<0$ for any $n$ and $\delta$ such that $\frac{\delta+k}{n}\geq \frac{d+k-l}{r}$, it's necessary to prove that
\begin{equation}
(r-n)k+nl<n(r-n)(g-1) \text{ for any } 1\leq n\leq r-1.
\label{eq:3.1}
\end{equation}

Thus we have following lemmas, which proof are easy and elementary (cf. \cite{NarasimhanRamanan}).
\begin{Lemma}
If $g\geq3$, $M$ contains $(0,1)$-stable and $(1,0)$-stable bundles. $M$ contains a $(1,1)$-stable bundle $W$ except $g=3, d, r$ both even.
\end{Lemma}

\begin{Lemma}
Let $0\rightarrow V\rightarrow W\rightarrow \mathcal{O}_{p}\rightarrow 0$ be an exact sequence, where $\mathcal{O}_{p}$ is the 1-dimensional skyscraper sheaf at $p\in C$. If $W$ is $(k,l)$-stable, then $V$ is $(k,l-1)$-stable.
\end{Lemma}

At first, let's recall a class of elliptic curves passing through a generic point, which are called \textbf{elliptic curves of Hecke type}. Let $U_{C}(r, d-1)$ be the moduli space of stable bundles of rank $r$ and degree $d-1$. Let $\mathfrak{D}\subset U_{C}(r, d-1)$
be the open set of (1,0)-stable bundles. Let $\psi: C\times \mathfrak{D}\rightarrow J^{d}(C)$ be defined as $\psi(x,V)=\mathcal{O}_{C}(x)\otimes \text{det}(V)$ and
$\mathfrak{R}_{C}:=\psi^{-1}(\mathcal{L})\subset C\times \mathcal{D}$ be the fibre of $\psi$ at the point $[\mathcal{L}]\in J^{d}(C)$. There exists a projective bundle
$$p:\mathfrak{ P}\rightarrow \mathfrak{R}_{C}$$
such that for any $(x, V)\in\mathfrak{R}_{C}$ we have $p^{-1}(x,V)=\mathbb{P}(V^{*}_{x})$. Let
$V^{*}_{x}\otimes \mathcal{O}_{\mathbb{P}(V^{*}_{x})}\rightarrow \mathcal{O}_{\mathbb{P}(V^{*}_{x})}(1)\rightarrow 0$ be the universal quotient, $f: C\times \mathbb{P}(V^{*}_{x})\rightarrow C$ be the projection, and
 $$0\rightarrow \mathfrak{E}^{*}\rightarrow f^{*}V^{*}\rightarrow _{\{x\}\times\mathbb{P}(V^{*}_{x})}\mathcal{O}_{\mathbb{P}(V^{*}_{x})}(1)\rightarrow 0$$
where $\mathfrak{E}^{*}$ is defined to the kernel of the surjection. Take dual, we have
\begin{equation}
0\rightarrow f^{*}V\rightarrow \mathfrak{E}\rightarrow _{\{x\}\times\mathbb{P}(V^{*}_{x})}\mathcal{O}_{\mathbb{P}(V^{*}_{x})}(-1)\rightarrow 0,
\label{eq:3.2}
\end{equation}
which, at any point $\xi=(V^{*}_{x}\rightarrow \Lambda \rightarrow 0)\in \mathbb{P}(V^{*})$, gives exact sequence
\[\begin{CD}
0@>>>V@>\iota>>\mathfrak{E}_{\xi}@>>>\mathcal{O}_{x}\rightarrow 0
\end{CD} \]
on $C$ such that $\text{ker}(\iota_{x})=\Lambda^{*}\subset V_{x}$. $V$ being (1,0)-stable implies stability of $\mathfrak{E}_{\xi}$. Thus (\ref{eq:3.2}) defines
\begin{equation}
\Psi_{(x,V)}: \mathbb{P}(V^{*}_{x})=p^{-1}(x,V)\longrightarrow M.
\label{eq:3.3}
\end{equation}

\begin{Definition}(cf. Definition 3.4 of \cite{Sun})
The images (under $\{\Psi_{(x,V)}\}_{(x,V)\in \mathfrak{R}_{C}}$) of lines in the fiber of $p: \mathfrak{P}\rightarrow \mathfrak{R}_{C}$ are the so called \textbf{Hecke curves} in $M$. The images (under $\{\Psi_{(x,V)}\}_{(x,V)\in \mathfrak{R}_{C}}$) of elliptic curves in the fibers of $p: \mathfrak{P}\rightarrow \mathfrak{R}_{C}$ are called \textbf{elliptic curves of Hecke type}.
\end{Definition}

It's known that (cf. Lemma 5.9 of \cite{NarasimhanRamanan}) that the morphisms in (\ref{eq:3.3}) are closed immersion, and the images of smooth elliptic curves $B\subset \mathbb{P}(V^{*}_{x})$ with degree 3 are smooth elliptic curves on $M$ that pass through generic point of $M$, which are elliptic curves of Hecke type and have degree $6r$(cf. Example 3.5 of \cite{Sun}).

If we do not require the curve $\phi: B\rightarrow M$ passing through generic point of $M$, there are elliptic curves with smaller degree. Now, let's recall a class elliptic curves passing through generic points of $M$, which are called \emph{elliptic curves of split type}.
For any given $r$ and $d$, let $r_{1}, r_{2}$ be positive integers and $d_{1}, d_{2}$ be integers that satisfy the equalities $r_{1}+r_{2}=r, d_{1}+d_{2}=d$ and
$$r_{1}\frac{d}{(r,d)}-d_{1}\frac{r}{(r,d)}=1, d_{2}\frac{r}{(r,d)}-r_{2}\frac{d}{(r,d)}=1.$$
Let $U_{C}(r_{1},d_{1})$ (resp. $U_{C}(r_{2},d_{2})$) be the moduli space of stable vector bundles with rank $r_{1}$ (resp. $r_{2}$) and degree $d_{1}$ (resp. $d_{2}$). Then, since $(r_{1},d_{1})=1$ and $(r_{2},d_{2})=1$, there are universal bundles $\mathcal{V}_{1}, \mathcal{V}_{2}$ on $C\times U_{C}(r_{1},d_{1})$ and $C\times U_{C}(r_{2},d_{2})$ respectively. Consider
\[\begin{CD}
U_{C}(r_{1},d_{1})\times U_{C}(r_{2},d_{2})@>\text{det}(\bullet)\times \text{det}(\bullet)>> J_{C}^{d_{1}}\times J_{C}^{d_{2}}@>(\bullet)\otimes (\bullet)>>J_{C}^{d},
\end{CD} \]
let $\mathcal{R}(r_{1},d_{1})$ be its fiber at $[\mathcal{L}]\in J_{C}^{d}$. The pullback of $\mathcal{V}_{1}, \mathcal{V}_{2}$ by the projection $C\times \mathcal{R}(r_{1}, d_{1})\rightarrow C\times U_{C}(r_{i},d_{i}) (i=1, 2)$ is still denoted by $\mathcal{V}_{1}, \mathcal{V}_{2}$ respectively. Let $p: C\times \mathcal{R}(r_{1}, d_{1})\rightarrow \mathcal{R}(r_{1}, d_{1})$ and
$\mathcal{G}=R^{1}p_{*}(\mathcal{V}_{2}^{*}\otimes \mathcal{V}_{1}),$
which is locally free of rank $r_{1}r_{2}(g-1)+(r,d)$. Let
$$q: P(r_{1}, d_{1})=\mathbb{P}(\mathcal{G})\rightarrow \mathcal{R}(r_{1}, d_{1})$$
be the projective bundle parametrizing 1-dimensional subspaces of $\mathcal{G}_{t} (t\in \mathcal{R}(r_{1}, d_{1}))$ and $f: C\times P(r_{1},d_{1})\rightarrow C$, $\pi: C\times P(r_{1},d_{1})\rightarrow P(r_{1},d_{1})$ be the projections. Then  there is a universal extension
\begin{equation}
0\rightarrow (id\times q)^{*}\mathcal{V}_{1}\otimes \pi^{*}\mathcal{O}_{P(r_{1},d_{1})}(1)\rightarrow \mathcal{E}\rightarrow (id\times q)^{*}\mathcal{V}_{2}\rightarrow 0
\label{eq:3.4}
\end{equation}
on $C\times P(r_{1}, d_{1})$ such that for any $x=([V_{1}], [V_{2}], [e])\in P(r_{1}, d_{1})$, where $[V_{i}]\in U_{C}(r_{i},d_{i})$ with $\text{det}(V_{1})\text{det}(V_{2})=\mathcal{L}$ and $[e]\subset H^{1}(C, V_{2}^{*}\otimes V_{1})$ being a line through a origin, the bundle $\mathcal{E}|_{C\times \{x\}}$ is the isomorphic class of vector bundles $V$ given by extensions
$$0\rightarrow  V_{1}\rightarrow V\rightarrow V_{2}\rightarrow 0$$
that defined by vectors on the line $[e]\subset H^{1}(C, V_{2}^{*}\otimes V_{1})$. Then $V$ must be stable( cf. Lemma 2.2 of \cite{MokSun09}), and the sequence (\ref{eq:3.4}) defines
$$\Phi: P(r_{1}, d_{1})\rightarrow SU_{C}(r,\mathcal{L})=M.$$
On each fiber $q^{-1}(\xi)=\mathbb{P}(H^{1}(V_{2}^{*}\otimes V_{1}))$ at $\xi=(V_{1},V_{2})$, the morphisms
\begin{equation}
\Phi_{\xi}:=\Phi|_{q^{-1}(\xi)}: q^{-1}(\xi)=\mathbb{P}(H^{1}(V_{2}^{*}\otimes V_{1}))\rightarrow M
\label{eq:3.5}
\end{equation}
is birational and $\Phi_{\xi}^{*}(-K_{M})=\mathcal{O}_{\mathbb{P}(H^{1}(V_{2}^{*}\otimes V_{1}))}(2(r,d))$ (cf. Lemma 2.4 of \cite{MokSun09}).

\begin{Definition}(cf. Example 3.6 of \cite{Sun})
The images (under $\{\Phi_{\xi}\}_{\xi\in \mathcal{R}(r_{1},d_{1})}$)of smooth elliptic curves in the fibers of $q: P(r_{1}, d_{1})=\mathbb{P}(\mathcal{G})\rightarrow \mathcal{R}(r_{1}, d_{1})$ are called \textbf{elliptic curves of split type}. For any smooth elliptic curve $B\subset q^{-1}(\xi)=\mathbb{P}(H^{1}(V_{2}^{*}\otimes V_{1}))$ of degree 3, the image of $\Phi_{\xi}|_{B}: B\rightarrow M$ is of degree $6(r,d)$, which is so called \textbf{elliptic curves of split type with minimal degree}.
\end{Definition}

 When $r=2$ and $d=1$, Sun has shown that any essential elliptic curve has degree at least 6, it has degree 6 if and only if it is an elliptic curve of split type with minimal degree. And then Sun conjectures the result holds for any rank $r$ and degree $d$ (i.e., \textbf{Sun's conjecture}). But when $r=3, d=1$, there is a class of elliptic curves in $M$ of degree 6, which are not the elliptic curves of split type.

Note that when $r=3$ and $ d=1$, we must have $r_{1}=1$ and $d_{1}=0$. Let $\mathcal{R}_{\mathcal{L}}:=\mathcal{R}(0,0)$ and let $ \mathcal{P}:=P(0,0)$.

\begin{Proposition}
Let $B\subset \mathcal{P}$ be an elliptic curve, which is mapped to a point in $J_{C}$  but is not in any fiber of $q$ and $\text{deg}(\mathcal{O}_{\mathcal{P}}(1)|_{B})=1$. If the normalization of the image of $B$ induced by $q$ is a line in $SU_{C}(2,\mathcal{L}')$ for some degree 1 line bundle $\mathcal{L}'$ on $C$ and $B$ is a degree 2 cover over its image in $SU_{C}(2,\mathcal{L}')$, then $\Phi|_{B}: B\rightarrow M$ is an essential elliptic curve on $M$ of degree 6.
\label{prop:3.1}
\end{Proposition}

\begin{proof}
Let $p_{1}: \mathcal{R}_{\mathcal{L}}\rightarrow J_{C}$, $p_{2}: \mathcal{R}_{\mathcal{L}}\rightarrow U_{C}(2,1)$ be the projections.
If the normalization of the image of $B$ induced by $q$ is a line in $SU_{C}(2,\mathcal{L}')$, then $p_{1}\circ q(B)=[\mathcal{L}\otimes \mathcal{L}'^{-1}]$ is a point of $J_{C}$ and let $L_{1}:=\mathcal{L}\otimes \mathcal{L}'^{-1}$. And then by the results of lines in \cite{Sun05} and \cite{MokSun09}, there are line bundles $L_{2}$ and $L_{3}$ on $C$ of degrees 0 and 1 respectively with $L_{2}\otimes L_{3}=\mathcal{L}'$, such that $B\rightarrow SU_{C}(2,\mathcal{L}')$ factors as the composition of $\varphi: B\rightarrow \mathbb{P}H^{1}(L_{3}^{-1}\otimes L_{2})$ with $\theta: \mathbb{P}H^{1}(L_{3}^{-1}\otimes L_{2})\rightarrow SU_{C}(2,\mathcal{L}')$ such that $\varphi^{*}\mathcal{O}_{\mathbb{P}H^{1}(L_{3}^{-1}\otimes L_{2})}=\mathcal{O}_{B}(2)$. Where $\theta: \mathbb{P}H^{1}(L_{3}^{-1}\otimes L_{2})\rightarrow SU_{C}(2,\mathcal{L}\otimes L_{1}^{-1})$ is a closed immersion defined by a vector bundle $\mathcal{E}'$ on $C\times \mathbb{P}H^{1}(L_{3}^{-1}\otimes L_{2})$ satisfying an exact sequence
\begin{equation}
0\rightarrow f^{*}L_{2}\otimes \pi^{*}\mathcal{O}_{\mathbb{P}H^{1}(L_{3}^{-1}\otimes L_{2})}(1)\rightarrow \mathcal{E}'\rightarrow f^{*}L_{3}\rightarrow 0.
\label{eq:3.6}
\end{equation}
By the construction of $q: \mathcal{P}\rightarrow \mathcal{R}_{\mathcal{L}}$, the restriction of (\ref{eq:3.4}) to $C\times B$ equals to
\begin{equation}
0\rightarrow f^{*}L_{1}\otimes \pi^{*}\mathcal{O}(1)\rightarrow E\rightarrow E'\rightarrow 0,
\label{eq:3.7}
\end{equation}
where $E':=(id_{C}\times\varphi)^{*}\mathcal{E}'$ satisfying
\begin{equation}
(0\rightarrow f^{*}L_{2}\otimes \pi^{*}\mathcal{O}_{B}(2)\rightarrow E'\rightarrow f^{*}L_{3}\rightarrow 0)\cong(id_{C}\times \varphi)^{*}(\ref{eq:3.6}).
\label{eq:3.8}
\end{equation}
Since $\Phi$ is defined by the sequence (\ref{eq:3.4}), so $\Phi|_{B}: B\rightarrow M$ is defined by the sequence (\ref{eq:3.7}). Thus $\text{deg}\Phi|_{B}^{*}(-K_{M})=\Delta(E)$. By considering sequences (\ref{eq:3.7}) and (\ref{eq:3.8}), we have $\Delta(E)=6$.
\end{proof}

 \begin{Remark}
 In fact, an elliptic curve in above proposition is defined by a vector bundle $E$ on $C\times B$ obtained by non-trivial extensions
\begin{equation}
0\rightarrow f^{*}L_{1}\otimes \pi^{*}\mathcal{O}(1)\longrightarrow E\longrightarrow E'\rightarrow 0,
\label{eq:1}
\end{equation}
\begin{equation}
0\rightarrow f^{*}L_{2}\otimes \pi^{*}\mathcal{O}(2)\longrightarrow E'\longrightarrow f^{*}L_{3}\rightarrow 0,
\label{eq:2}
\end{equation}
where $L_{1}, L_{2}$ are degree 0 line bundle on $C$ and $L_{3}$ is a degree 1 line bundle on $C$.

And, moreover, non-trivial extensions (\ref{eq:1}) and (\ref{eq:2}) exist.

Firstly, since
$$Ext^{1}(f^{*}L_{3}, f^{*}L_{2}\otimes \pi^{*}\mathcal{O}(2))\cong H^{1}(X, f^{*}(L_{3}^{-1}\otimes L_{2})\otimes \pi^{*}\mathcal{O}(2))$$
and $$ H^{1}(X, f^{*}(L_{3}^{-1}\otimes L_{2})\otimes \pi^{*}\mathcal{O}(2))\cong H^{1}(C, L_{3}^{-1}\otimes L_{2})\otimes H^{0}(B, {O}(2))$$
the second isomorphism because $H^{0}(C, L_{3}^{-1}\otimes L_{2})=0$, then we have
  $$\text{dim}Ext^{1}(f^{*}L_{3}, f^{*}L_{2}\otimes \pi^{*}\mathcal{O}(2))=2g,$$
So, there exist non-trivial extension(\ref{eq:2}).

Sencondly, on the one hand by the Hirzebruch-Riemann-Roch theorem, we have
\begin{equation}
\chi(E'^{*}\otimes f^{*}L_{1}\otimes\pi^{*}\mathcal{O}(1))=\text{deg}(ch(E'^{*}\otimes f^{*}L_{1}\otimes\pi^{*}\mathcal{O}(1)).td(T_{X}))_{2}=-1.
\label{eq:3}
\end{equation}
on the other hand,
\begin{equation}
\chi(E'^{*}\otimes f^{*}L_{1}\otimes \pi^{*}\mathcal{O}(1))=\Sigma_{i=0}^{2}h^{i}(X,E'^{*}\otimes f^{*}L_{1}\otimes \pi^{*}\mathcal{O}(1)).
\label{eq:4}
\end{equation}
Thus
 \begin{equation}
  \begin{aligned}
 ext^{1}(E', f^{*}L_{1}\otimes \pi^{*}\mathcal{O}(1))&=&h^{1}(X,E'^{*}\otimes f^{*}L_{1}\otimes \pi^{*}\mathcal{O}(1))\\
 &\geq & h^{2}(X,E'^{*}\otimes f^{*}L_{1}\otimes \pi^{*}\mathcal{O}(1))+1.
  \end{aligned}
 \label{eq:5}
 \end{equation}
By Serre duality, we have
\begin{equation}
H^{2}(X, E'^{*}\otimes f^{*}L_{1}\otimes \pi^{*}\mathcal{O}(1))\cong H^{0}(X,E'\otimes f^{*}L_{1}^{-1}\otimes\pi^{*}\mathcal{O}(-1)\otimes \omega_{X})^{\vee}.
\label{eq:7}
\end{equation}
Since $\Omega_{X}=f^{*}\Omega_{C}\oplus \pi^{*}\Omega_{B}$, we have $\omega_{X}=\text{det}\Omega_{X}=f^{*}\omega_{C}\otimes \pi^{*}\omega_{B}$.
Tensoring (\ref{eq:2}) by $(f^{*}L_{1}^{-1}\otimes\pi^{*}\mathcal{O}(-1)\otimes \omega_{X} )$, we have
 \begin{equation}\small
 0\rightarrow f^{*}(L_{2}\otimes L_{1}^{-1})\otimes \pi^{*}\mathcal{O}(1)\otimes\omega_{X} \rightarrow E'\otimes f^{*}L_{1}^{-1}\otimes\pi^{*}\mathcal{O}(-1)\otimes \omega_{X}\rightarrow f^{*}(L_{3}\otimes L_{1}^{-1})\pi^{*}\mathcal{O}(-1)\otimes \omega_{X}\rightarrow 0.
 \end{equation}
And, then taking cohomology, we have
\begin{equation}
0\rightarrow H^{0}(X,f^{*}(L_{2}\otimes L_{1}^{-1})\otimes \pi^{*}\mathcal{O}(1)\otimes\omega_{X})\rightarrow H^{0}(X,E'\otimes f^{*}L_{1}^{-1}\otimes\pi^{*}\mathcal{O}(-1)\otimes \omega_{X})\rightarrow \cdots.
\end{equation}
Which implies
\begin{equation}
\begin{aligned}
h^{0}(X,E'\otimes f^{*}L_{1}^{-1}\otimes\pi^{*}\mathcal{O}(-1)\otimes \omega_{X})&\geq&h^{0}(X,f^{*}(L_{2}\otimes L_{1}^{-1})\otimes \pi^{*}\mathcal{O}(1)\otimes\omega_{X})\\
&=&g-1+h^{0}(C,L_{2}^{-1}\otimes L_{1}).\ \ \ \ \ \ \ \ \ \ \ \ \ \
\end{aligned}
\label{eq:6}
\end{equation}
Then by (\ref{eq:5})(\ref{eq:6}) and (\ref{eq:7}), we have
\begin{equation}
 ext^{1}(E', f^{*}L_{1}\otimes \pi^{*}\mathcal{O}(1))\geq g+h^{0}(C,L_{2}^{-1}\otimes L_{1})\geq g\geq2,
\end{equation}
Thus, there exist non-trivial extensions (\ref{eq:1}).

\end{Remark}

\section{Minimal Elliptic Curves on Moduli Spaces}

In this section, we consider the moduli space $M$ of rank 3 stable bundles on $C$ with a fixed determinant $\mathcal{L}$ of degree 1. We also assume that the curve $C$ is generic in the sense that $C$ admits no surjective morphism to an elliptic curve. With this assumption, we know that $\text{Pic}(C\times B)=\text{Pic}(C)\times \text{Pic}(B)$ for any elliptic curve $B$.

For a morphism $\phi: B\rightarrow M$, it may happen that the normalization of $\phi(B)$ is a rational curve. To avoid this case, we assume that $\phi: B\rightarrow M$ is an essential elliptic curve of $M$ in this section. Let $E$ be the vector bundle on $X=C\times B$ that defines $\phi$. Consider the relative Harder-Narasimhan filtration (cf Theorem 2.3.2, page 45 in \cite{HuybrechtsLehn})
$$0=E_{0}\subset E_{1}\subset\cdots\subset E_{n}=E$$
over $C$. When $r=3$, there are three choices for $n: 1,  2 \text{ and } 3$.

\begin{Proposition} When $n=3$, we have $\Delta(E)\geq 10$. If $g\geq 3$ and $\phi: B\rightarrow M$ passes through a generic point of $M$, then $\Delta(E)\geq 18$.
\label{prop:4.1}
\end{Proposition}
\begin{proof} Let $0=E_{0}\subset E_{1}\subset E_{2}\subset E_{3}=E$ be the relative Harder-Narasimhan filtration over $C$. Let $F_{i}=E_{i}/E_{i-1} (i=1, 2, 3)$, then we have exact sequences
$$0\rightarrow E_{1}|_{X_{t}}\rightarrow E_{2}|_{X_{t}}\rightarrow F_{2}|_{X_{t}}\rightarrow 0 \quad\text{ and }\quad
0\rightarrow E_{2}|_{X_{t}}\rightarrow E|_{X_{t}}\rightarrow F_{3}|_{X_{t}}\rightarrow 0$$
on each fiber $X_{t}=\{t\}\times B$ of $f: X\rightarrow C$ since $\{F_{i}\}_{1=1, 2, 3}$ are flat over $C$. Thus $E_{1}$ is locally free ( cf Lemma 1.27 of \cite{Simpson}) and by the Theorem \ref{th:2.2}
\begin{equation}
\Delta(E)
=6c_{2}(F_{2})+6c_{2}(F_{3})+(2-6\text{deg}(E_{1}))(\mu_{1}-\mu_{2})+(4-6\text{deg}(E_{2}))(\mu_{2}-\mu_{3}).
\label{4.1}
\end{equation}
where $\mu_{i}=\mu(F_{i}|_{X_{t}}) (i=1, 2, 3)$ for generic $t\in C$.

Note that $c_{2}(F_{i})\geq 0 (i=2, 3)$ since $F_{i} (i=2, 3)$ are semi-stable on generic fiber of $f: X\rightarrow C$. Let $d_{i}=\text{deg}(E_{i})-\text{deg}(E_{i-1})$, then $d_{1}\leq 0$, $d_{1}+d_{2}\leq 0$ and $d_{1}+d_{2}+d_{3}=1$ since $E_{y}=E|_{C\times \{y\}}$ is stable of degree 1 for any $y\in B$.

If $\exists c_{2}(F_{i})\neq 0$ $(i=2 \text{ or }3)$, then $\Delta(E)=6c_{2}(F_{2})+6c_{2}(F_{3})+2(\mu_{1}-\mu_{2})+4(\mu_{2}-\mu_{3})\geq 12$. If $\phi: B\rightarrow M$ passes through a generic point, i.e., a (1,1)-stable point, which implies $\text{deg}(E_{1})\leq -1 \text{ and } \text{deg}(E_{2})\leq-1 $. Thus
$$\Delta(E)\geq 6+8(\mu_{1}-\mu_{2})+10(\mu_{2}-\mu_{3})\geq 24.$$

From now, we will assume that $c_{2}(F_{2})=c_{2}(F_{3})=0$. If $d_{1}\neq0$, we must have
$\Delta(E)\geq (2-6(-1))(\mu_{1}-\mu_{2})+(4-6\text{deg}(E_{2}))(\mu_{2}-\mu_{3})\geq 12.$
And, if $\phi: B\rightarrow M$ passes through a generic point, then
$$\Delta(E)\geq (2-6(-1))(\mu_{1}-\mu_{2})+(4-6(-1))(\mu_{2}-\mu_{3})\geq 18.$$

There left one case we need to consider when $c_{2}(F_{2})=c_{2}(F_{3})=0$ and $d_{1}=0$. In this case, we note that $\phi: B\rightarrow M$ can not passe through any generic point of $M$, $F_{2}$ and $F_{3}$ are line bundles and there are line bundles $\mathcal{L}_{1}, \mathcal{L}_{2}$ and $\mathcal{L}_{3}$ on $C$ of degrees 0, $d_{2}$ and $d_{3}$ respectively, such that
$$E_{1}=f^{*}\mathcal{L}_{1}\otimes \pi^{*}\mathcal{O}(\mu_{1}), \quad F_{2}=f^{*}\mathcal{L}_{2}\otimes \pi^{*}\mathcal{O}(\mu_{2}) \quad \text{ and }\quad F_{3}=f^{*}\mathcal{L}_{3}\otimes \pi^{*}\mathcal{O}(\mu_{3})$$
where $\mathcal{O}(\mu_{i})$ denote a line bundle of degree $\mu_{i}$ on $B$ ($i=1, 2, 3$). Replace $E$ by $E\otimes \pi^{*}\mathcal{O}(-\mu_{3})$, we can assume that $\mu_{3}=0$ and $\mu_{1}>\mu_{2}>0$. Now we have exact sequences
$$0\rightarrow f^{*}\mathcal{L}_{1}\otimes \pi^{*}\mathcal{O}(\mu_{1})\rightarrow E_{2}\rightarrow f^{*}\mathcal{L}_{2}\otimes \pi^{*}\mathcal{O}(\mu_{2})\rightarrow 0$$
and
$$0\rightarrow E_{2}\rightarrow E\rightarrow f^{*}\mathcal{L}_{3}\rightarrow 0.$$
Let $E':=E/(f^{*}\mathcal{L}_{1}\otimes \pi^{*}\mathcal{O}(\mu_{1}))$, then there is an induced morphism $\alpha: E'\rightarrow f^{*}\mathcal{L}_{3}$ satisfying the following commutative diagram
\[\begin{CD}
0@>>>f^{*}\mathcal{L}_{1}\otimes \pi^{*}\mathcal{O}(\mu_{1})@>>>E@>>>E'@>>>0\\
@.@VVV@|@VV\alpha V\\
0@>>>E_{2}@>>>E@>>>f^{*}\mathcal{L}_{3}@>>>0.
\end{CD} \]
By the Snake Lemma, $\alpha$ is surjective and $\text{ker}(\alpha)\cong f^{*}\mathcal{L}_{2}\otimes \pi^{*}\mathcal{O}(\mu_{2})$. Then $E'$ fits an exact sequence
$$ 0\rightarrow f^{*}\mathcal{L}_{2}\otimes \pi^{*}\mathcal{O}(\mu_{2})\rightarrow E'\rightarrow f^{*}\mathcal{L}_{3}\rightarrow 0$$
which induces a morphism
$$\psi: B\longrightarrow \mathbb{P}H^{1}(\mathcal{L}_{3}^{-1}\otimes \mathcal{L}_{2})=\mathbb{P}^{g-2+d_{3}-d_{2}}$$
such that $\psi^{*}\mathcal{O}_{\mathbb{P}^{g-2+d_{3}-d_{2}}}(1)=\mathcal{O}(\mu_{2})$. Thus $\mu_{2}\geq2$ and
$$\Delta(E)=2(\mu_{1}-\mu_{2})+4(\mu_{2}-\mu_{3})\geq 2+4\times 2=10.$$
\end{proof}

 When $n=2$, let $0\rightarrow E_{1}\rightarrow E\rightarrow F_{2}\rightarrow 0$ be the relative Harder-Narasimhan filtration over $C$, then we have an exact sequence
$$0\rightarrow E_{1}|_{X_{t}}\rightarrow E|_{X_{t}}\rightarrow F_{2}|_{X_{t}}\rightarrow 0$$
on each fiber $X_{t}=f^{-1}(t)$ of $f: X\rightarrow C$ since $E_{1}, F_{2} $ are flat over $C$. Thus $E_{1}$ is locally free (cf. Lemma 1.27 of \cite{Simpson}) and $F_{2}$ is locally free over $f^{-1}(C\setminus T)$ where $T\subset C$ is a finite set. Thus
\begin{equation}
0\rightarrow E_{1}|_{C\times \{y\}}\rightarrow E|_{C\times \{y\}}\rightarrow F_{2}|_{C\times \{y\}}\rightarrow 0, \text{ for any }y\in B
\label{eq:4.2}
\end{equation}
are exact sequences, which imply that $F_{2}$ is also $B-$flat.

When $n=2$, there are two cases: $\text{rk}E_{1}=2$ and $\text{rk}E_{1}=1$.

\begin{Lemma}
 If $g\geq 4$, $M$ contains (2, 0)-stable points.
\end{Lemma}

\begin{proof}
For $1\leq n\leq 3-1=2$ and $(k,l)=(2,0)$, the inequality (\ref{eq:3.1}) always holds when $g\geq 4$. Thus $M$ contains (2,0)-stable points when $g\geq 4$.
\end{proof}

\begin{Proposition} When $n=2$ and $\text{rk}E_{1}=2$, $\Delta(E)\geq 6$. If $g\geq 4$ and $\phi: B\rightarrow M$ passes through the generic points, $\Delta(E)\geq 20$.
\label{prop:4.2}
\end{Proposition}

\begin{proof}
In this case, tensoring $E$ with $\pi^{*}L$ for some suitable line bundle $L$ on $B$, we can assume that $\mu_{1}=0 \text{ or } \frac{1}{2}$, $\mu_{1}>\mu_{2}$ and
\begin{equation}
\Delta(E)=\frac{3}{2}\Delta(E_{1})+6c_{2}(F_{2})+(4-6\text{deg}E_{1})(\mu_{1}-\mu_{2}).
\label{eq:4.3}
\end{equation}

If $\Delta(E_{1})\neq 0$, then $\Delta(E_{1})=4c_{2}(E_{1})-2d_{1}(2\mu_{1})\geq 2$ by Theorem \ref{th:2.4} and the fact that $\text{Pic}(C\times B)=\text{Pic}(C)\times \text{Pic}(B)$. If $d_{1}=0$, then $\Delta(E_{1})=4c_{2}(E_{1})\in 4\mathbb{Z}$ and
$\Delta(E)\geq \frac{3}{2}\cdot 4+4\cdot\frac{1}{2}=8$.
If $d_{1}\leq -1$, then
$\Delta(E)\geq \frac{3}{2}\cdot 2+10\cdot\frac{1}{2}=8$.

When $\phi: B\rightarrow M$ passes through a generic point, i.e., a (2,0)-stable point, which implies $d_{1}=\text{deg}E_{1}\leq-1$.

If $\mu_{1}=0$, then $\Delta(E_{1})=4c_{2}(E_{1})-4d_{1}\mu_{1}\geq 4$. When $\text{deg}E_{1}\leq -2$, it's easy to see that $\Delta(E)\geq \frac{3}{2}\cdot 4+(4-6\cdot(-2))=22$.
Now we assume that $\text{deg}E_{1}=-1$. $E_{1y}$ is stable of degree -1 for generic $y\in B$ since $E_{y}$ is (2, 0)-stable. And we can prove that $\Delta(E_{1})\geq 8$ (same as the proof of Proposition 4.5 in \cite{Sun}), thus
$\Delta(E)\geq \frac{3}{2}\cdot 8+(4-6\cdot(-1))=22$.

If $\mu_{1}=\frac{1}{2}$, which means that $E_{1}$ is semi-stable of degree 1 at the generic fiber of $f: X\rightarrow C$. It's known that there is a unique stable rank 2 vector bundle with a fixed determinant of degree 1 on an elliptic curve. Thus $\Delta(E_{1})>0$ if and only if there exists $t_{1}\in C$ such that $E_{1t_{1}}=E_{1}|_{X_{t_{1}}}$ is not semi-stable. Let $E_{1t_{1}}\rightarrow \mathcal{O}(\mu_{1,1})\rightarrow 0$ be the quotient of minimal degree and
$$0\rightarrow E_{1}^{(1)}\rightarrow E_{1}\rightarrow_{X_{t_{1}}}\mathcal{O}(\mu_{1,1})\rightarrow 0$$
be the elementary transform of $E_{1}$ along $\mathcal{O}(\mu_{1,1})$ at $X_{t_{1}}$. If $E_{1}^{(i)}$ is defined and $\Delta(E_{1}^{(i)})>0$, let $t_{i+1}\in C$ such that $E_{1t_{i+1}}^{(i)}=E_{1}^{(i)}|_{X_{t_{i+1}}}$ is not semi-stable and $E_{1t_{i+1}}^{(i)}\rightarrow \mathcal{O}(\mu_{1,i+1})\rightarrow 0$ be the quotient of minimal degree, then we define $E_{1}^{(i+1)}$ to be the elementary transform of $E_{1}^{(i)}$ along $\mathcal{O}(\mu_{1,i+1})$ at $X_{t_{i+1}}$, namely $E_{1}^{(i+1)}$ satisfies the exact sequence
$$0\rightarrow E_{1}^{(i+1)}\rightarrow E_{1}^{(i)}\rightarrow _{X_{t_{i+1}}}\mathcal{O}(\mu_{1,i+1})\rightarrow 0.$$

Let $s_{1}$ be the minimal integer such that $\Delta(E_{1}^{(s_{1})})=0$. Then
\begin{equation}
\Delta(E_{1})=2s_{1}-4\sum_{i=1}^{s_{1}}\mu_{1,i},
\label{eq:4.3'}
\end{equation}
where $\mu_{1,i}\leq 0$. Same as the proof of Proposition 4.3 in\cite{Sun}, we can show that
\begin{equation}
s_{1}\geq \text{deg}E_{1}-2\text{deg}f_{*}E_{1}+2\text{dim}H^{0}(\mathcal{O}(\mu_{1,s_{1}})).
\label{eq:4.4}
\end{equation}
Hence, by (\ref{eq:4.3}), (\ref{eq:4.3'}) and (\ref{eq:4.4}), we have
\begin{equation}
\Delta(E)\geq
-6\text{deg}f_{*}E_{1}+6\text{dim}H^{0}(\mathcal{O}(\mu_{1,s_{1}}))-6\mu_{1,s_{1}}+2.
\label{eq:4.5}
\end{equation}

On the other hand, we claim that $\text{deg}f_{*}E_{1}\leq -2$ when $\phi: B\rightarrow M$ passes through the generic points, which means that $E_{y}$ is (2, 0)-stable for generic $y\in B$. To see it, we consider
$$0\rightarrow f^{*}f_{*}E_{1}\rightarrow E_{1}\rightarrow \mathcal{E}_{1}\rightarrow 0$$
where $\mathcal{E}_{1}$ is locally free over $f^{-1}(C\setminus T)$ and $T\subset C$ is a finite set such that $E_{1t}(t\in T)$ is not semi-stable. Thus, for any $y\in B$, the sequence
$$0\rightarrow (f^{*}f_{*}E_{1})_{y}\rightarrow E_{1y}\rightarrow \mathcal{E}_{1y}\rightarrow 0$$
is still exact, so we can consider $(f^{*}f_{*}E_{1})_{y}$ as a sub line bundle of $E_{y}$. Then since $E_{y}$ is (2, 0)-stable of degree 1 for generic $y\in B$,
$$\text{deg}f_{*}E_{1}+2=\text{deg}(f^{*}f_{*}E_{1})_{y}+2<\frac{\text{deg}E_{y}+2}{3},$$
which implies $\text{deg}f_{*}E_{1}\leq -2$. Therefore, if $\mu_{1,s_{1}}<0$, we have $\Delta(E)\geq 12+6+2=20$ by (\ref{eq:4.5}). If $\mu_{1,s_{1}}=0$, by tensoring $E$ with $\pi^{*}\mathcal{O}(-\mu_{1,s_{1}})$, we may assume $\text{dim}H^{0}(\mathcal{O}(\mu_{1,s_{1}}))=1$, then $\Delta(E)\geq 12+6+2=20$.

From now, we will assume that $\Delta(E_{1})=0$.

If $c_{2}(F_{2})\neq 0$, then $F_{2}$ is not locally free, which implies that there is a $y_{0}\in B$ such that $F_{2}|_{C\times \{y_{0}\}}$ has torsion $\tau(F_{2}|_{C\times \{y_{0}\}})\neq 0$ since $F_{2}$ is $B-$flat( cf Lemma 1.27 of \cite{Simpson}). Let
\begin{equation}
0\rightarrow \tau(F_{2}|_{C\times \{y_{0}\}})\rightarrow F_{2}|_{C\times \{y_{0}\}}\rightarrow F_{2}^{0}\rightarrow 0.
\label{eq:4.6}
\end{equation}
Then $F_{2}^{0}$ being a quotient line bundle of $E|_{C\times \{y_{0}\}}$ implies $\text{deg}F_{2}^{0}\geq 1$ since $E|_{C\times \{y_{0}\}}$ is stable. By sequences (\ref{eq:4.2} ) and  (\ref{eq:4.6}), we have
$$\text{deg}E_{1}|_{C\times \{y_{0}\}}=1-\text{deg}F_{2}^{0}-\text{dim}\tau(F_{2}|_{C\times \{y_{0}\}})\leq-1$$
which, by the formula (\ref{eq:4.3}), implies that
$$\Delta(E)\geq 6c_{2}(F_{2})+(4-6(-1))(\mu_{1}-\mu_{2})\geq 11.$$

When $\phi: B\rightarrow M$ passes through the generic points, means that $E_{y}=E|_{C\times \{y\}}$ is (2, 0)-stable for generic $y\in B$. Which implies that $\text{deg}E_{1}\leq -1$.

 If $\mu_{1}=0$. If $\text{deg}E_{1}\leq-2$, then $\Delta(E)\geq 6+(4-6(-2))=22$. Now we assume that $\text{deg}E_{1}=-1$. Note that $c_{2}(F_{2})\neq 0$ and $F_{2}$ being $C-$flat also imply that there exists a $t_{0}\in C$  such that $F_{2}|_{X_{t_{0}}}$ has torsion $\tau(F_{2}|_{X_{t_{0}}})\neq 0$. Let $0\rightarrow \tau(F_{2}|_{X_{t_{0}}})\rightarrow F_{2}|_{X_{t_{0}}}\rightarrow \mathcal{Q}\rightarrow 0$ and $E'=\text{ker}(E\rightarrow _{X_{t_{0}}}\mathcal{Q})$, then
$$0\rightarrow E'\rightarrow E\rightarrow _{X_{t_{0}}}\mathcal{Q}\rightarrow 0$$
which, for any $y\in B$, induces exact sequence
\begin{equation}
0\rightarrow E'|_{C\times\{y\}}\rightarrow E|_{C\times\{y\}}\rightarrow _{(t_{0}, y)} \mathcal{Q}\rightarrow 0.
\label{eq:4.7}
\end{equation}
Thus all $E'|_{C\times\{y\}}$ are semi-stable of degree 0. If $\phi: B\rightarrow M$ passes through the generic points, means that $E_{y}=E|_{C\times \{y\}}$ is (2, 0)-stable for generic $y\in B$, thus $E'_{y}$ is stable by (\ref{eq:4.7}). This implies that $\Delta(E')>0$. Otherwise $\{E'_{y}\}_{y\in B}$ are s-equivalent by applying Theorem\ref{th:2.4} to $\pi: X\rightarrow B$, which implies $E'=f^{*}V\otimes \pi^{*}L$ for a stable bundle $V$ on $C$ and a line bundle $L$ on $B$. Then $E_{t}=E'_{t}\cong L\oplus L\oplus L$ for any $t\neq t_{0}$, which is a contradiction since $E$ is not semi-stable on the generic fiber of $f: X\rightarrow C$.

To compute $\Delta(E')$, consider the relative Harder-Narasimhan filtration
$$0\rightarrow E'_{1}\rightarrow E' \rightarrow F'_{2}\rightarrow 0$$
over $C$, then $\mu(E'_{1}|_{X_{t}})=\mu_{1}$, $\mu(F'_{2}|_{X_{t}})=\mu_{2}$ for generic $t\in C$. Then
$$\Delta(E')=\frac{3}{2}\Delta(E'_{1})+6c_{2}(F'_{2})-6\text{deg}E'_{1}(\mu_{1}-\mu_{2})\geq 12.$$
To see it, we can assume $\Delta(E'_{1})=c_{2}(F'_{2})=0$ and $\text{deg}E'_{1}=-1$. Note that $E'_{1y}=E'_{1}|_{X_{y}}$ is stable of degree -1 for generic $y\in B$ since $E'_{y}$ is stable. Then, by Theorem \ref{th:2.4}, there are vector bundles $V'_{1}, V'_{2}$ on $C$ of rank 2, 1 respectively, and line bundles $\mathcal{O}(\mu_{i})$ of degree $\mu_{i} (i=1, 2)$ on $B$ such that $E'_{1}=f^{*}V'_{1}\otimes \pi^{*}\mathcal{O}(\mu_{1})$, $F'_{2}=f^{*}V'_{2}\otimes \pi^{*}\mathcal{O}(\mu_{2})$. Then we have
$$0\rightarrow f^{*}V'_{1}\otimes \pi^{*}\mathcal{O}(\mu_{1}-\mu_{2})\rightarrow E'\otimes \pi^{*}\mathcal{O}(-\mu_{2})\rightarrow f^{*}V'_{2}\rightarrow 0$$
which defines a morphism $\psi: B\rightarrow \mathbb{P}$ to a projective space such that $\pi^{*}\mathcal{O}(\mu_{1}-\mu_{2})=\psi^{*}\mathcal{O}_{\mathbb{P}}(1)$. Thus $\mu_{1}-\mu_{2}\geq2$ and $\Delta (E')\geq -6(-1)(\mu_{1}-\mu_{2})\geq 12$. Thus
$$\Delta(E)=\Delta(E')+6(\mu(E_{t_{0}})-\mu(\mathcal{Q}))\geq 12+6(\frac{2}{3}+1)=22.$$

If $\mu_{1}=\frac{1}{2}$, i.e., $E_{1}$ is semi-stable of degree 1 on the generic fiber of $f: X\rightarrow C$. By Theorem \ref{th:2.4}, $\Delta(E_{1})=0$ implies that there exist a stable rank 2 vector bundle $V$ of degree 1 on $B$ and a line bundle $L$ on $C$ such that $E_{1}=\pi^{*}V\otimes f^{*}L$. It is easy to see
$$\text{deg}E_{1}=2\text{deg}L\in 2\mathbb{Z}.$$
Moreover, we can show that $\text{deg}E_{1}\leq -4$. In fact, if $\text{deg}E_{1}=-2$, $E_{1y}$ is stable of degree -2 since $E_{y}$ is (2, 0)-stable for generic $y\in B$. Applying Theorem \ref{th:2.4}, $\Delta(E_{1})=0$ implies that there exist a stable rank 2 vector bundle $V_{1}$ of degree -2 on $C$ and a line bundle $L_{1}$ on $B$ such that $E_{1}=f^{*}V_{1}\otimes \pi^{*}L_{1}$ and $\text{deg}E_{1}|_{f^{-1}(t)}=2\text{deg}L_{1}$ for any $t\in C$, which imply a contradiction since $E_{1}|_{f^{-1}(t)}$ is semi-stable of degree 1 for generic $t\in C$. Thus $\text{deg}E_{1}\leq -4$ and $\Delta(E)\geq 6+(4-6(-4))\cdot \frac{1}{2}=20$.

Now we assume that $c_{2}(F_{2})=0$ and $\Delta(E_{1})=0$. The assumption $c_{2}(F_{2})=0$ implies $F_{2}$ is locally free and there is a line bundle $V_{2}$ on $C$ such that $F_{2}=f^{*}V_{2}\otimes \pi^{*}\mathcal{O}(\mu_{2})$. The degree formula becomes $$\Delta(E)=(4-6\text{deg}E_{1})(\mu_{1}-\mu_{2}).$$

If $\mu_{1}=0$, then $E_{1}$ is semi-stable of degree 0 at every fiber of $f: X\rightarrow C$ by applying Theorem \ref{th:2.4}. If $\text{deg}E_{1}\leq-1$, $\Delta(E)\geq (4-6(-1))=10$.
Now we consider the case $\text{deg}E_{1}=0$. In this case $E_{1}$ is semi-stable of degree 0 at every fiber of $\pi: X\rightarrow B$ since $E$ is stable of degree 1 at every fiber of $\pi: X\rightarrow B$. Apply Theorem \ref{th:2.4} to $f: X\rightarrow C$ (resp. $\pi: X\rightarrow B$), then $\Delta(E_{1})=0$ implies that $\{E_{1y}:=E_{1}|_{C\times \{y\}}\}_{y\in B}$ (resp.$\{E_{1t}:=E_{1}|_{\{t\}\times B}\}_{t\in C}$ ) are semi-stable and isomorphic to each other. By tensoring $E$ (thus $E_{1}$) with $\pi^{*}L$ (where $L$ is a line bundle of degree 0 on $B$), we can assume that $H^{0}(E_{1t})\neq0 (\forall t\in C)$, which has dimension at most 2 since $E_{1t}$ is semi-stable of degree 0. If $\text{dim}H^{0}(E_{1t})=2$, then $f_{*}E_{1}$ is a degree 0 vector bundle of rank 2 on $C$, $E_{1}=f^{*}f_{*}E_{1}$ and $E\otimes \pi^{*}\mathcal{O}(-\mu_{2})$ fits an exact sequence
$$0\rightarrow f^{*}f_{*}E_{1}\otimes \pi^{*}\mathcal{O}(\mu_{1}-\mu_{2})\rightarrow E\otimes \pi^{*}\mathcal{O}(-\mu_{2})\rightarrow f^{*}V_{2}\rightarrow 0$$
which defines a morphism $\psi: B\rightarrow \mathbb{P}$ to a projective space $\mathbb{P}$ such that $\psi^{*}\mathcal{O}_{\mathbb{P}}(1)=\mathcal{O}(\mu_{1}-\mu_{2})$, thus $\mu_{1}-\mu_{2}\geq 2$ and $\Delta(E)\geq 4\cdot 2=8$. If $\text{dim}H^{0}(E_{1t})=1$, $f_{*}E_{1}$ is a line bundle on $C$ since $\{E_{1t}\}_{t\in C}$ are isomorphic each other. Then we have an exact sequence
$$0\rightarrow f^{*}f_{*}E_{1}\rightarrow E_{1}\rightarrow f^{*}V_{1}\otimes \pi^{*}L_{1}\rightarrow 0$$
for a line bundle $V_{1}$ on $C$ and a degree 0 line bundle $L_{1}$ on $B$. Consider $f^{*}f_{*}E_{1}$ as a sub line bundle of $E$ and let $E':=E/(f^{*}f_{*}E_{1})$, there is an induced morphism $\beta: f^{*}V_{1}\otimes \pi^{*}L_{1}\rightarrow E'$ satisfying the diagram
\[\begin{CD}
0@>>>f^{*}f_{*}E_{1}@>>>E_{1}@>>>f^{*}V_{1}\otimes \pi^{*}L_{1}@>>>0\\
@.@|@VVV@VV\beta V\\
0@>>>f^{*}f_{*}E_{1}@>>>E@>>>E'@>>>0.
\end{CD} \]
By the snake lemma, $\beta$ is injective and $\text{Coker}(\beta)=f^{*}V_{2}\otimes \pi^{*}\mathcal{O}(\mu_{2})$. Then $E'\otimes \pi^{*}\mathcal{O}(-\mu_{2})$ fits an exact sequence
\begin{equation}
0\rightarrow f^{*}V_{1}\otimes \pi^{*}\mathcal{O}(-\mu_{2})\rightarrow E'\otimes \pi^{*}\mathcal{O}(-\mu_{2})\rightarrow f^{*}V_{2}\rightarrow 0,
\label{eq:4.8}
\end{equation}
where $\mathcal{O}(-\mu_{2})$ is a line bundle of degree $-\mu_{2}$ on $B$. (\ref{eq:4.8}) defines a morphism $\varphi': B\rightarrow \mathbb{P}H^{1}(V_{2}^{-1}\otimes V_{1})=\mathbb{P}^{g-1}$ such that $\varphi'^{*}\mathcal{O}_{\mathbb{P}^{g-1}}(1)=\mathcal{O}(-\mu_{2})$, thus $-\mu_{2}\geq2$ and $\Delta(E)\geq 4\cdot 2=8$.

  If $\phi: B\rightarrow M$ passes through the generic points, then $E_{y}$ is (2, 0)-stable for generic $y\in B$. Then $\text{deg}E_{1}\leq -1$. If $\text{deg}E_{1}\leq-3$, it's easy to see that $\Delta(E)\geq 4-6(-3)=22$. If $\text{deg}E_{1}= -1 (\text{ or }-2)$, then $E_{1y}$ is stable of degree $-1$ (or $-2$) for generic $y\in B$ since $E_{y}$ is (2, 0)-stable, which implies that all the bundles $\{E_{1y}=E_{1}|_{C\times \{y\}}\}_{y\in B}$ are stable and isomorphic each other by applying Theorem \ref{th:2.4} to $\pi: X\rightarrow B$. Then there is a stable vector bundle $V_{1}$ of rank 2 on $C$ and a line bundle $\mathcal{O}(\mu_{1})$ on $B$ such that $E_{1}=f^{*}V_{1}\otimes \pi^{*}\mathcal{O}(\mu_{1})$. Then $E\otimes \pi^{*}\mathcal{O}(-\mu_{2})$ fits an exact sequence
$$0\rightarrow f^{*}V_{1}\otimes \pi^{*}\mathcal{O}(\mu_{1}-\mu_{2})\rightarrow E\otimes \pi^{*}\mathcal{O}(-\mu_{2})\rightarrow f^{*}V_{2}\rightarrow 0$$
which defines a morphism $\varphi: B\rightarrow \mathbb{P}'$ to a projective space $\mathbb{P}'$ such that $\varphi^{*}\mathcal{O}_{\mathbb{P}'}(1)=\mathcal{O}(\mu_{1}-\mu_{2})$, thus $\mu_{1}-\mu_{2}\geq 2$ and $\Delta(E)\geq(4-6(-1))\cdot 2=20$.

 If $\mu_{1}=\frac{1}{2}$, then $E_{1}$ is semi-stable of degree 1 at generic fiber of $f: X\rightarrow C$. By Applying Theorem \ref{th:2.4}, $\Delta(E_{1})=0$ implies $\{E_{1t}=E_{1}|_{\{t\}\times B}\}_{t\in C}$ are semi-stable of degree 1 and $s-$equivalent, thus they are stable and isomorphic. Then $E_{1}=\pi^{*}V \otimes f^{*}L$ for a stable bundle $V$ of degree 1 on $B$ and a line bundle $L$ on $C$, which implies that $\text{deg}E_{1}= 2\text{deg}L\in 2\mathbb{Z}$ and $f_{*}E_{1}=f_{*}(\pi^{*}V \otimes f^{*}L)\cong f_{*}\pi^{*}V \otimes L\cong L$. Then we have an exact
$$0\rightarrow f^{*}L\rightarrow E_{1}\rightarrow f^{*}L\otimes \pi^{*}\mathcal{O}(1)\rightarrow 0$$
for a degree 1 line bundle $\mathcal{O}(1)$ on $B$. Consider $f^{*}L$ as a sub line bundle of $E$ and let $E':=E/(f^{*}L)$, same as above, we have an exact sequence
$$0\rightarrow f^{*}L\otimes \pi^{*}\mathcal{O}(1-\mu_{2})\rightarrow E'\otimes \pi^{*}\mathcal{O}(-\mu_{2})\rightarrow f^{*}V_{2}\rightarrow 0$$
and $1-\mu_{2}\geq 2$. Thus $\mu_{2}\leq-1$ and $\Delta(E)\geq4(\frac{1}{2}-(-1))=6$.

 If $\phi: B\rightarrow M$ passes through the generic points, $\text{deg}E_{1}\leq-2$ since $\text{deg}E_{1}\in2\mathbb{Z}$ and $E_{y}$ is (2,0)-stable for generic $y\in B$. Thus $\Delta(E)\geq(4-6(-2))(\frac{1}{2}-(-1))=24$.
\end{proof}

\begin{Lemma}
If $g\geq4$, $M$ contains (1,2)-stable points. If $g>4$, $M$ contains (3,1)-stable points.
\end{Lemma}

\begin{Proposition}
 When $n=2$ and $\text{rk}E_{1}=1$, $\Delta(E)\geq6$. If $g>4$ and $\phi:B\rightarrow M$ passes through the generic points, then $\Delta(E)\geq 18$.
\label{prop:4.3}
\end{Proposition}

\begin{proof}
In this case, $E_{1}=f^{*}V_{1}\otimes \pi^{*}\mathcal{O}(\mu_{1})$ for a line bundle $V_{1}$ on $C$ and a degree $\mu_{1}$ line bundle $\mathcal{O}(\mu_{1})$ on $B$. The degree formula becomes
\begin{equation}
\Delta(E)=\frac{3}{2}\Delta(F_{2})+(2-6\text{deg}E_{1})(\mu_{1}-\mu_{2}).
\label{eq:4.9}
\end{equation}
Tensoring $E$ with $\pi^{*}L$ for some suitable line bundle $L$ on $B$, we can assume that $\mu_{2}=0 \text{ or }\frac{1}{2}$.

We consider the case $\Delta(F_{2})=0$ at first, which implies that $F_{2}$ is locally free and $F_{2}$ is semi-stable of slop $\mu_{2}$ at every fiber of $f: X\rightarrow C$.

If $\mu_{2}=0$, then $F_{2}$ is semi-stable of degree 0 at generic fiber of $f: X\rightarrow C$. By applying Theorem \ref{th:2.4}, $\Delta(F_{2})=0$ implies all the bundles $\{F_{2y}:=F_{2}|_{C\times\{y\}}\}_{y\in B}$ are isomorphic to each other. If $\text{deg}E_{1}\leq-1$, it's easy to see $\Delta(E)\geq(2-6(-1))(\mu_{1}-\mu_{2})\geq8$.
 If $\text{deg}E_{1}=0$, $F_{2y}$ is stable of degree 1 for any $y\in B$, then there is a stable vector bundle $V_{2}$ of degree 1 on $C$ and a degree $\mu_{2}$ line bundle $\mathcal{O}(\mu_{2})$ on $B$ such that $F_{2}=f^{*}V_{2}\otimes\pi^{*}\mathcal{O}(\mu_{2})$, and $E\otimes\mathcal{O}(-\mu_{2})$ satisfies an exact sequence
 $$0\rightarrow f^{*}V_{1}\otimes\pi^{*}\mathcal{O}(\mu_{1}-\mu_{2})\rightarrow E\otimes\mathcal{O}(-\mu_{2})\rightarrow f^{*}V_{2}\rightarrow 0$$
 which defines a morphism
 $$\varphi: B\rightarrow \mathbb{P}H^{1}(V_{2}^{-1}\otimes V_{1})\cong\mathbb{P}^{2g-2}$$
 such that$\varphi^{*}\mathcal{O}_{\mathbb{P}^{2g-2}}(1)=\mathcal{O}(\mu_{1}-\mu_{2})$ and $\phi: B\rightarrow M$ factors through $\varphi: B\rightarrow \mathbb{P}H^{1}(V_{2}^{-1}\otimes V_{1})\cong\mathbb{P}^{2g-2}$, which implies that the normalization of $\varphi(B)$ is an elliptic curve. Hence $\mu_{1}-\mu_{2}\geq3$ and $\Delta(E)\geq2\cdot3=6$.

 When $\phi: B\rightarrow M$ passes through a generic point, i.e., there is a $y_{0}\in B$ such that $E_{y_{0}}$ is (1,1)-stable, which implies that $\text{deg}E_{1}\leq-1$. If $\text{deg}E_{1}\leq-3$, we have $\Delta(E)\geq(2-6(-3))(\mu_{1}-\mu_{2})\geq20$. When $\text{deg}E_{1}=-1 \text{ or }-2$, $F_{2y_{0}}$ is stable of degree 2( or 3) since $E_{y_{0}}$ is (1,1)-stable. And since all the bundles $\{F_{2y}:=F_{2}|_{C\times \{y\}}\}_{y\in B}$ are isomorphic to each other, then there is a stable vector bundle $V_{2}$ of degree 2( or 3) on $C$ and a degree $\mu_{2}$ line bundle $\mathcal{O}(\mu_{2})$ on $B$ such that $F_{2}=f^{*}V_{2}\otimes\pi^{*}\mathcal{O}(\mu_{2})$. Same as above, we can see that $\mu_{1}-\mu_{2}\geq3$ and $\Delta(E)\geq(2-6(-1))\cdot3=24$.

If $\mu_{2}=\frac{1}{2}$. By applying Theorem \ref{th:2.4}, $\Delta(F_{2})=0$ implies all the bundles $\{F_{2t}=F_{2}|_{\{t\}\times B}\}_{t\in C}$ are semi-stable of degree 1 and $s-$equivalent, thus they are stable and isomorphic to each other. Then $F_{2}=\pi^{*}V \otimes f^{*}L_{2}$ for a stable bundle $V$ of degree 1 on $B$ and a line bundle $L_{2}$ on $C$, which implies that $\text{deg}F_{2}= 2\text{deg}L_{2}\in 2\mathbb{Z}$ and
$$f^{*}f_{*}F_{2}=f^{*}f_{*}(\pi^{*}V \otimes f^{*}L_{2})\cong f^{*}(f_{*}\pi^{*}V\otimes L_{2})\cong f^{*}L_{2}.$$
Then we have exact sequence
$$0\rightarrow f^{*}f_{*}F_{2}=f^{*}L_{2}\rightarrow F_{2}\rightarrow f^{*}L_{2}\otimes \pi^{*}\mathcal{O}(1)\rightarrow 0$$
for a degree 1 line bundle $\mathcal{O}(1)$ on $B$. Consider $f^{*}L_{2}\otimes \pi^{*}\mathcal{O}(1)$ as a quotient line bundle of $E$ and let $E':=\text{ker}(E\rightarrow f^{*}L_{2}\otimes \pi^{*}\mathcal{O}(1))$, there is a induced morphism $\gamma: E'\rightarrow f^{*}f_{*}F_{2}=f^{*}L_{2}$ satisfying the diagram
\[\begin{CD}
0@>>>E'@>>>E@>>>f^{*}L_{2}\otimes \pi^{*}\mathcal{O}(1)@>>>0\\
@.@V\gamma VV@VVV@|\\
0@>>>f^{*}L_{2}@>>>F_{2}@>>>f^{*}L_{2}\otimes \pi^{*}\mathcal{O}(1)@>>>0.
\end{CD} \]
By the snake lemma, $\gamma$ is surjective and $\text{ker}\gamma= E_{1}=f^{*}V_{1}\otimes \pi^{*}\mathcal{O}(\mu_{1})$. Then $E'$ fits an exact sequence
$$0\rightarrow f^{*}V_{1}\otimes \pi^{*}\mathcal{O}(\mu_{1})\rightarrow E'\rightarrow f^{*}L_{2}\rightarrow 0$$
which defines a morphism
$$\varphi_{E'}: B\rightarrow \mathbb{P}H^{1}(L_{2}^{-1}\otimes V_{1})\cong\mathbb{P}^{\frac{1-3d_{1}}{2}+g-2}$$
such that $\varphi_{E'}^{*}\mathcal{O}_{\mathbb{P}H^{1}(L_{2}^{-1}\otimes V_{1})}(1)=\mathcal{O}(\mu_{1})$. Thus $\mu_{1}\geq 2$. On the other hand, since $E$ is stable of degree 1 at every fiber of $\pi: X\rightarrow B$ and $\text{deg}F_{2}= 2\text{deg}L_{2}\geq1$, we have $\text{deg}F_{2}\geq 2$. Hence, $\text{deg}E_{1}\leq-1$ and
$$\Delta(E)\geq(2-6(-1))(2-\frac{1}{2})=12.$$
 If $\phi: B\rightarrow M$ passes a generic point, there is a $y_{0}\in B$ such that $E_{y_{0}}$ is (1,1)-stable, which implies
$$\text{deg}L_{2}-1>\frac{\text{deg}E+1-1}{3}=\frac{1}{3}.$$
So $\text{deg}F_{2}= 2\text{deg}L_{2}\geq4$ and $\text{deg}E_{1}\leq -3$.
Thus
 $$\Delta(E)\geq(2-6(-3))(2-(\frac{1}{2}))=30.$$

 From now we will consider the case that $\Delta(F_{2})=4c_{2}(F_{2})-c_{1}(F_{2})^{2}\neq0$.

 We consider the case that $F_{2}$ is locally free at first.

 If $\mu_{2}=0$, then $F_{2}$ is semi-stable of degree 0 at the generic fiber of $f: X\rightarrow C$. By Theorem \ref{th:2.4}, $\Delta(F_{2})\neq 0$ implies $\Delta(F_{2})>0$. On the other hand, $c_{1}(F_{2})^{2}=0$ since $F_{2}$ is semi-stable of degree 0 at the generic fiber of $f: X\rightarrow C$ and $\text{Pic}(C\times B)=\text{Pic}(C)\times \text{Pic}B$. Thus $\Delta(F_{2})=4c_{2}(F_{2})\geq 4$ and
 $$\Delta(E)\geq \frac{3}{2}\cdot 4+2\cdot (\mu_{1}-\mu_{2})\geq 8.$$
 When $\phi: B\rightarrow M$ passes through a generic point, i.e., a (1, 1)-stable point, we have $\text{deg}E_{1}\leq -1$. If $\text{deg}E_{1}\leq -2$, it's easy to see $\Delta(E)\geq \frac{3}{2}\cdot 4+(2-6(-2))(\mu_{1}-\mu_{2})\geq 20$.
Now we assume that $\text{deg}E_{1}=-1$, then $F_{2y}$ is stable of degree 2 for generic $y\in B$ since $E_{y}$ is (1, 1)-stable. We claim that $\Delta(F_{2})\geq 8$, which implies $\Delta(E)\geq\frac{3}{2}\cdot 8+(2-6(-1))(\mu_{1}-\mu_{2})\geq 20$ by (\ref{eq:4.9}).

 We prove the above claim as following: If $F_{2}$ is semi-stable of degree 0 at every fiber $f: X\rightarrow C$, then $F_{2}$ induces a non-trivial morphism $\varphi_{F_{2}}: C\rightarrow \mathbb{P}^{1}$ (cf. \cite{FriedmanMorganWitten}) such that $\varphi_{F_{2}}^{*}\mathcal{O}_{\mathbb{P}^{1}}(1)=(\text{det}f_{!}F_{2})^{-1}$ which has degree $c_{2}(F_{2})$ by Grothendieck-Riemann-Roch theorem. Then $$\Delta(F_{2})=4c_{2}(F_{2})=4\text{deg}\varphi_{F_{2}}\geq 8.$$
If there exists a $t_{0}\in C$ such that $F_{2t_{0}}=F_{2}|_{f^{-1}(t_{0})}$ is not semi-stable, let $F_{2t_{0}}\rightarrow \mathcal{O}(\mu)\rightarrow 0$ be the quotient line bundle of minimal degree $\mu<0$ and $F'_{2}=\text{ker}(F_{2}\rightarrow _{X_{t_{0}}}\mathcal{O}(\mu)\rightarrow 0)$. If $\Delta(F'_{2})\neq 0$, then $\Delta(F'_{2})=4c_{2}(F'_{2})\geq 4$ and $\Delta(F_{2})=\Delta(F'_{2})-4\mu\geq8$ by Lemma \ref{lm2.6}. If $\Delta(F'_{2})=0$, all the bundles $\{F'_{2y}=F_{2}|_{C\times\{y\}}\}_{y\in B}$ are isomorphic to each other by applying Theorem\ref{th:2.4} to $f: X\rightarrow C$. On the other hand, by the definition of $F'_{2}$, we have exact sequences
 \begin{equation}
 0\rightarrow F'_{2}\rightarrow F_{2}\rightarrow _{X_{t_{0}}}\mathcal{O}(\mu)\rightarrow 0
 \label{eq:4.9'}
 \end{equation}
 and
 $$0\rightarrow F'_{2y}\rightarrow F_{2y}\rightarrow_{(t_{0},y)}\mathbb{C}\rightarrow 0 \text{ for any }y\in B.$$
 Then $F'_{2y}$ is stable of degree 1 for generic $y\in B$ since $F_{2y}$ is stable of degree 2. Thus all the bundles $\{F'_{2y}=F_{2}|_{C\times\{y\}}\}_{y\in B}$ are stable of degree 1 and isomorphic to each other, then $F'_{2}=f^{*}V'_{2}\otimes \pi^{*}L'$ for a degree 1 stable vector bundle $V'_{2}$ on $C$ and a degree 0 line bundle $L'$ on $B$. Then (\ref{eq:4.9'}) induces a non-trivial morphism $\psi: B\rightarrow \mathbb{P}(V^{'*}_{2t_{0}})$ such that $\mathcal{O}(-\mu)=\psi^{*}\mathcal{O}_{\mathbb{P}(V^{'*}_{2t_{0}})}(1)$. Thus $-\mu\geq 2$ and $\Delta(F_{2})\geq 8$.

 If $\mu_{2}=\frac{1}{2}$, then $F_{2}$ is semi-stable of degree 1 at the generic fiber of $f: X\rightarrow C$.
It's known that there is a unique  stable rank 2 vector bundle with fixed determinant of degree 1 on an elliptic curve, and note that $F_{2}$ is locally free. Thus $\Delta(F_{2})>0$ if and only if there exists $t_{1}\in C$ such that $F_{2t_{1}}:=F_{2}|_{\{t_{1}\}\times B}$ is not semi-stable.

 Let $F_{2t_{1}}\rightarrow \mathcal{O}(\mu_{2,1})\rightarrow 0$ be the quotient of minimal degree and
 $$0\rightarrow F_{2}^{(1)}\rightarrow F_{2}\rightarrow _{X_{t_{1}}}\mathcal{O}(\mu_{2,1})\rightarrow 0$$
 be the elementary transformation of $F_{2}$ along $\mathcal{O}(\mu_{2,1})$ at $X_{t_{1}}$. If $F_{2}^{(i)}$ is defined and $\Delta(F_{2}^{(i)})>0$, let $t_{i+1}\in C$ such that $F_{2t_{i+1}}^{(i)}:=F_{2}^{(i)}|_{X_{t_{i+1}}}$ is not semi-stable and $F_{2t_{i+1}}^{(i)}\rightarrow \mathcal{O}(\mu_{2,i+1})\rightarrow 0$ be the quotient of minimal degree, then we define $F_{2}^{(i+1)}$ to be the elementary transform of $F_{2}^{(i)}$ along $\mathcal{O}(\mu_{2,i+1})$ at $X_{t_{i+1}}$, namely $F_{2}^{(i+1)}$ satisfies the exact sequence
\begin{equation}
 0\rightarrow F_{2}^{(i+1)}\rightarrow F_{2}^{(i)}\rightarrow _{X_{t_{i+1}}}\mathcal{O}(\mu_{2,i+1})\rightarrow 0.
 \label{eq:4.10}
 \end{equation}
Let $s_{2}$ be the minimal integer such that $\Delta(F_{2}^{(s_{2})})=0$. Then
$$\Delta(F_{2})=\Delta(F_{2}^{(s_{2})})+\sum_{i=1}^{s_{2}}(2-4\mu_{2,i})=2s_{2}-4\sum_{i=1}^{s_{2}}\mu_{2,i},$$
where $\mu_{2,i}\leq0 (i=1,2,\cdots, s_{2})$. Take direct image of (\ref{eq:4.10}), we have
$$0\rightarrow f_{*}F_{2}^{(s_{2})}\rightarrow f_{*}F_{2}^{(s_{2}-1)}\rightarrow _{t_{s_{2}}}H^{0}(\mathcal{O}(\mu_{2,s_{2}}))\rightarrow R^{1}f_{*}F_{2}^{(s_{2})}=0$$
 and $\text{deg}f_{*}F_{2}^{(i+1)}\leq\text{deg}f_{*}F_{2}^{(i)}$, which implies
\begin{equation}
\text{deg}f_{*}F_{2}^{(s_{2})}\leq \text{deg}f_{*}F_{2}-\text{dim}H^{0}(\mathcal{O}(\mu_{2,s_{2}})).
\label{eq:4.11}
\end{equation}
Restrict (\ref{eq:4.10}) to a fiber $X_{y}=\pi^{-1}(y)$, we have exact sequence
$$0\rightarrow F_{2y}^{(i+1)}\rightarrow F_{2y}^{(i)}\rightarrow _{(t_{i+1},y)}\mathbb{C}\rightarrow 0$$
which implies that
\begin{equation}
\text{deg}F_{2y}^{(s_{2})}=\text{deg}F_{2y}^{(s_{2}-1)}-1=\cdots=\text{deg}F_{2y}-s_{2}=\text{deg}F_{2}-s_{2}.
\label{eq:4.12}
\end{equation}
On the other hand, by Theorem \ref{th:2.4}, $\Delta(F_{2}^{(s_{2})})=0$ implies that there exists a stable rank 2 vector bundle $V$ of degree 1 on $B$ and a line bundle $L$ on $C$ such that $F_{2}^{(s_{2})}=\pi^{*}V\otimes f^{*}L$. It's easy to see
$$\text{deg}F_{2y}^{(s_{2})}=2\text{deg}L=2\text{deg}f_{*}F_{2}^{(s_{2})}.$$
Thus combine (\ref{eq:4.11}) and (\ref{eq:4.12}), we have the inequality
\begin{equation}
s_{2}\geq 1-\text{deg}E_{1}-2\text{deg}f_{*}F_{2}+2\text{dim}H^{0}(\mathcal{O}(\mu_{2,s_{2}})).
\label{eq:4.13}
\end{equation}
We claim that $\text{deg}f_{*}F_{2}\leq-\text{deg}E_{1}$. To see it, consider
\begin{equation}
0\rightarrow \mathcal{F}'_{2}:=f^{*}f_{*}F_{2}\rightarrow F_{2}\rightarrow \mathcal{F}_{2}\rightarrow 0
\label{eq:4.14}
\end{equation}
where $\mathcal{F}_{2}$ is locally free on $f^{-1}(C\setminus T)$ and $T\subset C$ is a finite set such that $F_{2t}(t\in T)$ is not semi-stable. Thus, for any $y\in B$ , the sequence
\begin{equation}
0\rightarrow \mathcal{F}'_{2y}\rightarrow F_{2y}\rightarrow \mathcal{F}_{2y}\rightarrow 0
\label{eq:4.15}
\end{equation}
is still exact, which implies that $\mathcal{F}_{2}$ is $B-$flat(cf Lemma 2.1.4 of \cite{HuybrechtsLehn}). The sequence (\ref{eq:4.15}) and (\ref{eq:4.2}) already imply $\text{deg}f_{*}F_{2}=\text{deg}\mathcal{F}'_{2y}\leq-\text{deg}E_{1}$ since $E_{y}$ is stable of degree 1(Choose $y\in B$ such that $\mathcal{F}_{2y}$ is free, then $\text{deg}\mathcal{F}_{2y}=\mu(\mathcal{F}_{2y})>\mu(E_{y})=\frac{1}{3}$ and $\text{deg}f_{*}F_{2}=\text{deg}\mathcal{F}'_{2y}=\text{deg}F_{2y}-\text{deg}\mathcal{F}_{2y} \leq 1-\text{deg}E_{1} -1=-\text{deg}E_{1}$.).
Thus
\begin{eqnarray*}
\Delta(E)&=&\frac{3}{2}\Delta(F_{2})+(2-6\text{deg}E_{1})(\mu_{1}-\mu_{2})\\
&=&\frac{3}{2}(2s_{2}-4\sum_{i=1}^{s_{2}}\mu_{2,i})+(2-6\text{deg}E_{1})(\mu_{1}-\mu_{2})\\
&\geq&3(1-\text{deg}E_{1}-2\text{deg}f_{*}F_{2}+2\text{dim}H^{0}(\mathcal{O}(\mu_{2,s_{2}}))-2\sum_{i=1}^{s_{2}}\mu_{2,i})+(2-6\text{deg}E_{1})(\mu_{1}-\mu_{2})\\
&\geq&3(1-\text{deg}E_{1}+2\text{deg}E_{1}+2\text{dim}H^{0}(\mathcal{O}(\mu_{2,s_{2}}))-2\sum_{i=1}^{s_{2}}\mu_{2,i})+(2-6\text{deg}E_{1})\frac{1}{2}\\
&\geq&4+6\text{dim}H^{0}(\mathcal{O}(\mu_{2,s_{2}}))-6\mu_{2,s_{2}}.
\end{eqnarray*}
If $\mu_{2,s_{2}}<0$, then $\Delta(E)\geq4-6(-1)=10$. If $\mu_{2,s_{2}}=0$, tensoring $E$ with $\pi^{*}\mathcal{O}(-\mu_{2,s_{2}})$£¬ we can assume that $\text{dim}H^{0}(\mathcal{O}(\mu_{2,s_{2}}))\neq0$ and $\Delta(E)\geq4+6=10$.

When $\phi: B\rightarrow M$ passes the generic points, i.e., the (1,2)-stable points, the sequences (\ref{eq:4.15}) and (\ref{eq:4.2}) also imply that $\text{deg}f_{*}F_{2}\leq -2-\text{deg}E_{1}$ since $E_{y}$ is (1,2)-stable of degree 1 for generic $y\in B$. Thus $\Delta(E)\geq 22$.

 Now we consider the case that $F_{2}$ is not locally free, which implies there exists a $y_{0}\in B$ such that $F_{2}|_{X_{y_{0}}}$ has torsion $\tau(F_{2}|_{X_{y_{0}}})\neq 0$ since $F_{2}$ is $B-$flat(cf Lemma 1.27 of \cite{Simpson}). Let
\begin{equation}
0\rightarrow \tau(F_{2}|_{X_{y_{0}}})\rightarrow F_{2}|_{X_{y_{0}}}\rightarrow F_{2}^{0}\rightarrow 0.
\label{eq:4.16}
\end{equation}
Then $F_{2}^{0}$ being quotient bundle of $E|_{X_{y_{0}}}$ implies that
$$\frac{\text{deg}F_{2}^{0}}{2}=\mu(F_{2}^{0})>\mu(E|_{X_{y_{0}}})=\frac{1}{3} \Longrightarrow\text{deg}F_{2}^{0}\geq 1$$
since $E|_{X_{y_{0}}}$ is stable of degree 1. By sequences (\ref{eq:4.2}) and (\ref{eq:4.16}),
we have
$$\mu(E_{1})=\text{deg}E|_{X_{y_{0}}}=1-\text{deg}F_{2}^{0}-\text{dim}\tau(F_{2}|_{X_{y_{0}}})\leq-1,$$
which, by the formula (\ref{eq:4.9}), we have
\begin{eqnarray*}
\Delta(E)=\frac{3}{2}\Delta(F_{2})+(2-6\text{deg}E_{1})(\mu_{1}-\mu_{2})
\geq\frac{3}{2}\cdot 2+(2-6(-1))\frac{1}{2}=7.
\end{eqnarray*}
When $\phi: B\rightarrow M$ passes through a generic point, in order to show $\Delta(E)\geq18$, we note that $F_{2}$ being not locally free and $C$-flat also imply that there is a $t_{0}\in C$ such that $F_{2}|_{X_{t_{0}}}$ has non-trivial torsion $\tau(F_{2}|_{X_{t_{0}}})\neq0$. Let $0\rightarrow \tau(F_{2}|_{X_{t_{0}}})\rightarrow F_{2}|_{X_{t_{0}}}\rightarrow \mathcal{Q}\rightarrow 0$ and $E'=\text{ker}(E\rightarrow _{X_{t_{0}}}\mathcal{Q}\rightarrow 0)$, then
$$0\rightarrow E'\rightarrow E\rightarrow _{X_{t_{0}}}\mathcal{Q}\rightarrow 0$$
which, for any $y\in B$, induces exact sequence
\begin{equation}
0\rightarrow E'_{y}\rightarrow E_{y}\rightarrow _{(t_{0},y)}\mathbb{C}^{2}\rightarrow 0.
\label{eq:4.16'}
\end{equation}
Thus all $E'_{y}:=E'|_{C\times\{y\}}$ is of degree -1. If $\phi: B\rightarrow M$ passes through generic points, ie., $E_{y}$ is (1,2)-stable for generic points $y\in B$, then it's easy to see that $E'_{y}$ is stable for generic $y\in B$ by (\ref{eq:4.16'}). This implies that $\Delta(E')>0$. Otherwise $\{E'_{y}\}_{y\in B}$ are s-equivalent by applying Theorem \ref{th:2.4} to $\pi: X\rightarrow B$, which implies $E'=f^{*}V'\otimes \pi^{*}L$ for a stable bundle $V$ on $C$ and a line bundle $L$ on $B$. Then $E_{t}=E'_{t}=L\oplus L\oplus L$ for any $t\neq t_{0}$, which is a contradiction since $E$ is not semi-stable on the generic fiber of $f: X\rightarrow C$.

To compute $\Delta(E')$, we consider the relative Harder-Narasimhan filtration
$$0\rightarrow E'_{1}\rightarrow E'\rightarrow F'_{2}\rightarrow 0$$
over $C$, then $\mu(E_{1}'|_{X_{t}})=\mu_{1}$ and $\mu(F'_{2}|_{X_{t}})=\mu_{2}$ for generic $t\in C$. Then
\begin{equation}
\Delta(E')=\frac{3}{2}\Delta(F'_{2})+(-2-6\text{deg}E'_{1})(\mu_{1}-\mu_{2}).
\label{eq:4.16"}
\end{equation}

If $\mu_{2}=0$, then $F'_{2}$ is semi-stable of degree 0 at generic fiber of $f: X\rightarrow C$. We can prove that $\Delta(E')\geq 8$.
It's easy to see that $\Delta(E')\geq(-2-6(-2))(\mu_{1}-\mu_{2})=10$ when $\text{deg}E'_{1}\leq-2$, now we assume that $\text{deg}E'_{1}=-1$ and then $F'_{2y}$ is semi-stable of degree 0 for generic $y\in B$ since $E'_{y}$ is stable of degree -1. If $\Delta(F'_{2})\neq 0$, then $\Delta(F'_{2})=4c_{2}(F'_{2})\geq 4$ by Theorem \ref{th:2.4} and the fact that $c_{1}(F'_{2})^{2}=0$ since $\text{Pic}(C\times B)=\text{Pic}(C)\times \text{Pic}(B)$. Thus $\Delta(E')=\frac{3}{2}\Delta(F'_{2})+(-2-6\text{deg}E'_{1})(\mu_{1}-\mu_{2})\geq \frac{3}{2}\cdot 4+(-2-6(-1))=10$.
If $\Delta(F'_{2})=0$, by applying Theorem \ref{th:2.4} to $f: X\rightarrow C$ and $\pi: X\rightarrow B$, all the bundles $\{F'_{2t}:=F'_{2}|_{\{t\}\times B}\}_{t\in C}$ are semi-stable and isomorphic to each other. Tensoring $E$ (thus $E'$) with $\pi^{*}L$ (for a degree 0 line bundle $L$ on $B$), we can assume that $H^{0}(F'_{2t})\neq 0$ (for any $t\in C$), which has dimension at most 2 since $F'_{2t}$ is semi-stable of degree 0. If $\text{dim}H^{0}(F'_{2t})=2$, then $f_{*}F'_{2}=V'_{2}$ is a vector bundle of rank 2 and $F'_{2}=f^{*}V'_{2}$. Thus $E'$ satisfies an exact sequence
$$0\rightarrow f^{*}V'_{1}\otimes \pi^{*}\mathcal{O}(\mu_{1})\rightarrow E'\rightarrow f^{*}V'_{2}\rightarrow 0$$
for a line bundle $V'_{1}$ on $C$ and a degree $\mu_{1}$ line bundle $\mathcal{O}(\mu_{1})$ on $B$. Which induces a non-trivial morphism $\varphi_{E'}: B\rightarrow \mathbb{P}$ to a projective space $\mathbb{P}$ such that $\mathcal{O}(\mu_{1})=\varphi_{E'}^{*}\mathcal{O}_{\mathbb{P}}(1)$, thus $\mu_{1}\geq 2$ and $\Delta(E')\geq(-2-6(-1))\cdot 2=8$. If $\text{dim}H^{0}(F'_{2t})=1$, then $V'_{2}:=f_{*}F'_{2}$ is a line bundle and we have an exact sequence
$$0\rightarrow f^{*}V'_{2}\rightarrow F'_{2}\rightarrow f^{*}V'_{3}\otimes \pi^{*}\mathcal{O}(\mu_{2})\rightarrow 0$$
for a line bundle $V'_{3}$ on $C$ and a degree $\mu_{2}=0$ line bundle $\mathcal{O}(\mu_{2})$ on $B$. Consider $f^{*}V'_{3}\otimes \pi^{*}\mathcal{O}(\mu_{2})$ as a quotient line bundle of $E'$ and let $E":=\text{ker}(E'\rightarrow f^{*}V'_{3}\otimes \pi^{*}\mathcal{O}(\mu_{2})\rightarrow 0)$, then there is a induced morphism $\alpha":E"\rightarrow f^{*}V'_{2}$ satisfies the following diagram
\[\begin{CD}
0@>>>E"@>>>E'@>>>f^{*}V'_{3}\otimes \pi^{*}\mathcal{O}(\mu_{2})@>>>0\\
@.@V\alpha" VV@VVV@|\\
0@>>>f^{*}V'_{2}@>>>F'_{2}@>>>f^{*}V'_{3}\otimes \pi^{*}\mathcal{O}(\mu_{2})@>>>0.
\end{CD} \]
 By the snake lemma, $\alpha"$ is surjective and $\text{ker}(\alpha")=E'_{1}=f^{*}V'_{1}\otimes \pi^{*}\mathcal{O}(\mu_{1})$. Thus $E"$ satisfies an exact sequence
$$0\rightarrow f^{*}V'_{1}\otimes \pi^{*}\mathcal{O}(\mu_{1})\rightarrow E"\rightarrow f^{*}V'_{2} \rightarrow 0,$$
which induces a morphism $\varphi_{E"}: B\rightarrow \mathbb{P}"$ to a projective space $\mathbb{P}"$ such that $\mathcal{O}(\mu_{1})=\varphi_{E"}^{*}\mathcal{O}_{\mathbb{P}"}(1)$. Thus $\mu_{1}\geq2$ and $\Delta(E')\geq(-2-6(-1))\cdot 2=8$. Hence $\Delta(E')\geq 8$ when $\mu_{2}=0$. Then
$$\Delta(E)=\Delta(E')+6(\mu(E_{t_{0}})-\mu(\mathcal{Q}))\text{rk}\mathcal{Q}\geq 8+12(\frac{1}{3}+\frac{1}{2})=18.$$

 If $\mu_{2}=\frac{1}{2}$, then $F'_{2}$ is semi-stable of degree 1 at generic fiber of $f: X\rightarrow C$.
 If $\Delta(F'_{2})=0$, then $F'_{2}=\pi^{*}V'\otimes f^{*}L'_{2}$ for a rank 2 stable bundle $V'$ on $B$ and a line $L'_{2}$ on $C$. So $f_{*}F'_{2}=f_{*}(\pi^{*}V'\otimes f^{*}L'_{2})\cong f_{*}\pi^{*}V'\otimes L'_{2}\cong L'_{2}$ is a line bundle and we have an exact sequence
$$0\rightarrow f^{*}L'_{2}\rightarrow F'_{2}\rightarrow f^{*}L'_{2}\otimes \pi^{*}\mathcal{O}(1)\rightarrow 0$$
for a degree 1 line bundle $\mathcal{O}(1)$ on $B$. Same as above, we can show that $\mu_{1}\geq2$ and $\Delta(E')\geq(-2-6(-1))(2-\frac{1}{2})=6$. Then $\Delta(E)=\Delta(E')+6(\mu(E_{t_{0}})-\mu(\mathcal{Q}))\text{rk}\mathcal{Q}\geq 6+12(\frac{1}{2}+\frac{1}{2})=18$.
If $\Delta(F'_{2})=4c_{2}(F'_{2})-c_{1}(F'_{2})^{2}\neq0$, then $\Delta(F'_{2})\geq 2$ by Theorem 2.4 and $c_{1}(F'_{2})^{2}=2\text{deg}F'_{2}(2\mu_{2})$ since $\text{Pic}(C\times B) =\text{Pic}(C)\times \text{Pic}(B)$. If $\phi: B\rightarrow M$ passes through the generic points when $g>4$, then $E_{y}$ is (3,1)-stable for generic $y\in B$, which implies that $\text{deg}E'_{1}\leq -3$. Thus $\Delta(E')\geq\frac{3}{2}\cdot 2+(-2-6(-3))(\mu_{1}-\mu_{2}) \geq 11$ and
$$\Delta(E)=\Delta(E')+6(\mu(E_{t_{0}})-\mu(\mathcal{Q}))\text{rk}\mathcal{Q} \geq 11+12(\frac{1}{6}+\frac{1}{2})=19.$$

\end{proof}

Now we consider the case that n=1, i.e., $E$ is semi-stable on the generic fiber of $f: X\rightarrow C$. Tensoring $E$ with $\pi^{*}L$ for a suitable line bundle $L$ on $B$, we can assume that $0\leq\text{deg}(E|_{X_{t}})\leq 2$ on $X_{t}=f^{-1}(t)$.

\begin{Proposition}
When $E$ is semi-stable of degree 1 on the generic fiber of $f: X\rightarrow C$, we have $\Delta(E)\geq14$. If $g\geq 4$ and when $\phi: B\rightarrow M$ passes through a generic point, then $\Delta(E)\geq 20$.
\label{prop:4.4}
\end{Proposition}

\begin{proof}
It's known that there is a unique stable rank 3 vector bundle with a fixed determinant of degree 1 on an elliptic curve. Thus $\Delta(E)>0$ if and only if there exists $t_{1}\in C$ such that $E_{t_{1}}=E|_{X_{t_{1}}}$ is not semistable.

Let $E_{t_{1}}\rightarrow G_{1}\rightarrow 0$ be a indecomposable quotient of minimal slop and
$$0\rightarrow E^{(1)}\rightarrow E\rightarrow _{X_{t_{1}}}G_{1}\rightarrow 0$$
be the elementary transformation of $E$ along $G_{1}$ at $X_{t_{1}}$. If $E^{(i)}$ is defined and $\Delta(E^{(i)})>0$, let $t_{i+1}\in C$ such that $E^{(i)}_{t_{i+1}}=E^{(i)}|_{X_{t_{i+1}}}$ is not semi-stable and $E^{(i)}_{t_{i+1}}\rightarrow G_{i+1}\rightarrow 0$ be a indecomposable quotient of minimal slop, then we define $E^{(i+1)}$ to be the elementary transformation of $E^{(i)}$ along $G_{i+1}$ at $X_{t_{i+1}}$, namely $E^{(i+1)}$ satisfies the exact sequence
\begin{equation}
0\rightarrow E^{(i+1)}\rightarrow E^{(i)}\rightarrow _{X_{t_{i+1}}}G_{i+1}\rightarrow 0.
\label{eq:4.17}
\end{equation}

Let $s$ be the minimal integer such that $\Delta(E^{(s)})=0$, and let

$$s_{1}=\sharp\{i:\text{rk}G_{i}=1 \text{ but }i\neq s\} \quad\text{ and }\quad s_{2}=\sharp\{i:\text{rk}G_{i}=2 \text{ but }i\neq s\}.$$
Then
$$s_{1}+s_{2}+1=s \quad\text{ and } \quad s_{1}+2s_{2}+\text{rk}G_{s}=\sum_{i=1}^{s}\text{rk}G_{i},$$
and
\begin{equation}
\Delta(E)=\sum_{i=1}^{s}6(\frac{1}{3}-\mu(G_{i}))\text{rk}G_{i}
\geq 2s_{1}+4s_{2}+6(\frac{1}{3}-\mu(G_{s}))\text{rk}G_{s},
\label{eq:4.18}
\end{equation}
where $\mu(G_{i})\leq 0(i=1, 2, \cdots, s)$. Take direct image of (\ref{eq:4.17}), we have
\begin{equation}
0\rightarrow f_{*}E^{(s)}\rightarrow f_{*}E^{(s-1)}\rightarrow _{t_{s}}H^{0}(G_{s})\rightarrow 0
\label{eq:4.19}
\end{equation}
(since $R^{1}f_{*}E^{(s)}=0$) and $\text{deg}f_{*}E^{(i+1)}\leq \text{deg}f_{*}E^{(i)}$, which imply
\begin{equation}
\text{deg}f_{*}E^{(s)}\leq \text{deg}f_{*}E-\text{dim}H^{0}(G_{s}).
\label{eq:4.20}
\end{equation}
Restrict (\ref{eq:4.17}) to a fiber $X_{y}=\pi^{-1}(y)$, we have exact sequence
$$0\rightarrow E^{(i+1)}_{y}\rightarrow E^{(i)}_{y}\rightarrow _{(t_{i+1},y)}\mathbb{C}^{\text{rk}G_{i+1}}\rightarrow 0,$$
which implies that
\begin{equation}
\text{deg}E^{(s)}_{y}=\text{deg}E^{(s-1)}_{y}-\text{rk}G_{s}=\cdots=\text{deg}E_{y}-\sum_{i=1}^{s}\text{rk}G_{i}.
\label{eq:4.21}
\end{equation}
On the other hand, by Theorem\ref{th:2.4}, $\Delta(E^{(s)})=0$ implies that there exists a stable vector bundle $V$ of rank 3 and degree 1 on $B$ and a line bundle $L$ on $C$ such that $E^{(s)}=\pi^{*}V\otimes f^{*}L$. It's easy to see
$$\text{deg}E^{(s)}_{y}=3\text{deg}L=3\text{deg}f_{*}E^{(s)}.$$
Thus combine (\ref{eq:4.20}) and (\ref{eq:4.21}), we have the inequality
\begin{equation}
\sum_{i=1}^{s}\text{rk}G_{i}\geq 1-3\text{deg}f_{*}E+3\text{dim}H^{0}(G_{s}).
\label{eq:4.22}
\end{equation}

To see $\Delta(E)\geq14$, consider the exact sequence
\begin{equation}
0\rightarrow \mathcal{F}'=f^{*}f_{*}E\rightarrow E\rightarrow \mathcal{F}\rightarrow 0
\label{eq:4.23}
\end{equation}
where $\mathcal{F}$ is locally free on $f^{-1}(C\setminus T)$ and $T\subset C$ is a finite set such that $E_{t}(t\in T)$ is not semi-stable. Thus, $\forall y\in B$, the sequence
\begin{equation}
0\rightarrow \mathcal{F}'_{y}\rightarrow E_{y}\rightarrow \mathcal{F}_{y}\rightarrow 0
\label{eq:4.24}
\end{equation}
is still exact, which implies $\mathcal{F}$ is $B-$flat( cf Lemma 2.1.4 of \cite{HuybrechtsLehn}). The sequence (\ref{eq:4.24}) already implies $\text{deg}f_{*}E=\text{deg}\mathcal{F}'_{y}\leq0$ since $E_{y}$ is stable of degree 1.

If $\text{deg}f_{*}E=\text{deg}\mathcal{F}'_{y}=0$, then $\mathcal{F}_{y}$ is stable of degree 1 and $\mathcal{F}$ is locally free (Otherwise, there is at least a $y_{0}\in B$ such that $\mathcal{F}_{y_{0}}$ has torsion(cf Lemma 1.27 of \cite{Simpson}). The stability of $E_{y_{0}}$ implies that $\mathcal{F}_{y_{0}}/torsion$ has degree at least 1. Thus $\text{deg}\mathcal{F}_{y_{0}}\geq 2$ and $\text{deg}f_{*}E=\text{deg}\mathcal{F}'_{y_{0}}\leq-1$, which contradicts to the assumption that $\text{deg}f_{*}E=0$.). On the other hand, by the definition of $\mathcal{F}$, we know that $\mathcal{F}$ is semi-stable of degree 1 on the generic fiber of $f: X\rightarrow C$. This implies $\Delta(\mathcal{F})>0$.
( Otherwise, $\{\mathcal{F}_{t}\}_{t\in C}$ are semi-stable of degree 1 and s-equivalent by applying Theorem \ref{th:2.4} to $f: X\rightarrow C$, which implies $\mathcal{F}=\pi^{*}V'\otimes f^{*}L'$ for a stable bundle $V'$ on $B$ and a line bundle $L'$ on $C$. Then $\text{deg}\mathcal{F}_{y}=2\text{deg}L'$ which contradict to that $\mathcal{F}_{y}$ is of degree 1. )
Then, same as the proof of Proposition 4.3 of \cite{Sun}, we can prove that $\Delta(\mathcal{F})\geq 10$. On the other hand, by (\ref{eq:4.23}), we have
$$\Delta(\mathcal{F})=4c_{2}(\mathcal{F})-2(1-\text{deg}f_{*}E)=4c_{2}(\mathcal{F})-2$$
which implies $c_{2}(\mathcal{F})\geq 3$ and
$$\Delta(E)=6\text{deg}f_{*}E+6c_{2}(\mathcal{F})-4=6c_{2}(\mathcal{F})-4\geq 14.$$

Now we assume that $\text{deg}f_{*}E\leq -1$, which means that $\sum_{i=1}^{s}\text{rk}G_{i}\geq 4+3\text{dim}H^{0}(G_{s})$.
When $\text{rk}G_{s}=2$, if $\mu(G_{s})<0$, then $\Delta(E)\geq \text{Min}_{s_{1}+2s_{2}+2\geq 4}\{2s_{1}+4s_{2}+6(\frac{1}{3}-(-\frac{1}{2}))\cdot 2\}\geq14$ by (\ref{eq:4.18}); if $\mu(G_{s})=0$, tensoring $E$ with a suitable line bundle $\pi^{*}L^{-1}$, we can assume that $H^{0}(G_{s})\neq 0$ (cf. Theorem 5 of \cite{Atiyah}), then $\Delta(E)\geq \text{Min}_{s_{1}+2s_{2}+2\geq 7}\{2s_{1}+4s_{2}+6\cdot\frac{1}{3}\cdot 2\}\geq14$.
 Now we assume that $\text{rk}G_{s}=1$, if $\mu(G_{s})<0$, then $\Delta(E)\geq \text{Min}_{s_{1}+2s_{2}+1\geq4}\{2s_{1}+4s_{2}+6(\frac{1}{3}-(-1))\}\geq14$; if $\mu(G_{s})=0$, tensoring $E$ with $\pi^{*}G_{s}^{-1}$, we can assume that $H^{0}(G_{s})\neq 0$, then $\Delta(E)\geq\text{Min}_{s_{1}+2s_{2}+1\geq7} \{2s_{1}+4s_{2}+6\cdot\frac{1}{3}\}\geq 14$.

If $\phi: B\rightarrow M$ passes through a generic point, we claim that $\text{deg}f_{*}E\leq-2$, which implies $\Delta(E)\geq 20$. To prove the claim, we will prove that $\phi(B)$ lies in a proper closed subset if $\text{deg}f_{*}E=-1 \text{ or }0$. If $\text{deg}f_{*}E=-1$, note that $\mathcal{F}_{y}$ must be locally free of degree 2 for generic $y\in B$(if $\mathcal{F}_{y}$ has nontrivial torsion, then $E_{y}$ has a quotient bundle of rank 2 and degree at most 1, which is impossible since $E_{y}$ is (1,1)-stable for generic $y\in B$). Thus $E_{y}$ satisfies $0\rightarrow V_{1}\rightarrow E_{y}\rightarrow V_{2}\rightarrow 0$ where $V_{1}, V_{2}$ are vector bundles on $C$ of ranks 1, 2 and degrees -1, 2 respectively such that $V_{1}\otimes \text{det}V_{2}=\mathcal{L}$. To estimate the dimension of the locus of such bundles, we can assume that both $V_{1}$ and $V_{2}$ is stable. The locus of such bundles has dimension at most $g+4(g-1)+1+h^{1}(V_{2}^{*}\otimes V_{1})-1-g=6(g-1)+4<\text{dim}M$ when $g\geq4$. Similarly, if $\text{deg}f_{*}E=0$, we can show that $\phi(B)$ lies in a locus of dimension at most $6(g-1)+1<\text{dim}M$.
\end{proof}

\begin{Lemma}
If $g\geq 12$, $M$ contains (1,10)-stable points.
\end{Lemma}

\begin{Proposition}
If $E$ is semi-stable of degree 2 at the generic fiber of $f: X\rightarrow C$, $\Delta(E)\geq 8$. If $g>12$ and $\phi: B\rightarrow M$ passes through the generic points, $\Delta(E)\geq 18$.
\label{prop:4.5}
\end{Proposition}

\begin{proof}
It's known that there is a unique stable rank 3 vector bundle with a fixed determinant of degree 2 on an elliptic curve. Thus $\Delta(E)>0$ if and only if there exists $t_{1}\in C$ such that $E_{t_{1}}=E|_{X_{t_{1}}}$ is not semistable.

Let $E_{t_{1}}\rightarrow G_{1}\rightarrow 0$ be a indecomposable quotient of minimal slop and
$$0\rightarrow E^{(1)}\rightarrow E\rightarrow _{X_{t_{1}}}G_{1}\rightarrow 0$$
be the elementary transformation of $E$ along $G_{1}$ at $X_{t_{1}}$. If $E^{(i)}$ is defined and $\Delta(E^{(i)})>0$, let $t_{i+1}\in C$ such that $E^{(i)}_{t_{i+1}}=E^{(i)}|_{X_{t_{i+1}}}$ is not semi-stable and $E^{(i)}_{t_{i+1}}\rightarrow G_{i+1}\rightarrow 0$ be a indecomposable quotient of minimal slop, then we define $E^{(i+1)}$ to be the elementary transformation of $E^{(i)}$ along $G_{i+1}$ at $X_{t_{i+1}}$, namely $E^{(i+1)}$ satisfies the exact sequence
$$0\rightarrow E^{(i+1)}\rightarrow E^{(i)}\rightarrow _{X_{t_{i+1}}}G_{i+1}\rightarrow 0.$$

Let $s$ be the minimal integer such that $\Delta(E^{(s)})=0$ and
\begin{equation}
\Delta(E)=\sum_{i=1}^{s}6(\frac{2}{3}-\mu(G_{i}))\text{rk}G_{i},
\label{eq:4.24.5}
\end{equation}
where $\mu(G_{i})\leq\frac{1}{2}$ if $\text{rk}G_{i}=2$ and $\mu(G_{i})\leq0$ if $\text{rk}G_{i}=1$.
Let
$$s_{1}=\sharp\{i:\text{rk}G_{i}=1 \text{ but }i\neq s\} \text{ and } s_{2}=\sharp\{i:\text{rk}G_{i}=2 \text{ but }i\neq s\}.$$
Then
$$s_{1}+s_{2}+1=s \text{ and } s_{1}+2s_{2}+\text{rk}G_{s}=\sum_{i=1}^{s}\text{rk}G_{i},$$
and
\begin{equation}
\Delta(E)=\sum_{i=1}^{s}6(\frac{2}{3}-\mu(G_{i}))\text{rk}G_{i}
\geq 4s_{1}+2s_{2}+6(\frac{2}{3}-\mu(G_{s}))\text{rk}G_{s},
\label{eq:4.25}
\end{equation}

Same as the above proposition, we have
\begin{equation}
\text{deg}f_{*}E^{(s)}\leq \text{deg}f_{*}E-\text{dim}H^{0}(G_{s})
\label{eq:4.25.1}
\end{equation}
and
\begin{equation}
\text{deg}E^{(s)}=1-\sum_{i=1}^{s}\text{rk} G_{i}.
\label{eq:4.25.2}
\end{equation}
On the other hand, by Theorem 2.4, $\Delta(E^{(s)})=0$ implies that there is a stable vector bundle $V$ of  rank 3 and degree 2 on $B$ and a line bundle $L$ on $C$ such that $E^{(s)}=\pi^{*}V\otimes f^{*}L$. It's easy to see
\begin{equation}
\text{deg}E^{(s)}_{y}=3\text{deg}L
\label{eq:4.25.3}
\end{equation}
 and
 \begin{equation}
 \text{deg}f_{*}E^{(s)}=2\text{deg}L.
\label{eq:4.25"}
\end{equation}
Thus
\begin{equation}
2\sum_{i=1}^{s}\text{rk}(G_{i})\geq 2-3\text{deg}f_{*}E+3\text{dim}H^{0}(G_{s}).
\label{eq:4.25'}
\end{equation}

We claim that $\text{deg}f_{*}E\leq -1$. To show it, consider
$$0\rightarrow \mathcal{F}'=f^{*}f_{*}E\rightarrow E\rightarrow \mathcal{F}\rightarrow 0$$
where $\mathcal{F}$ is locally free of rank 1 on $f^{-1}(C\setminus T)$ and $T\subset C$ is a finite set such that $E_{t}(t\in T)$ is not semi-stable. Thus, for any $y\in B$, the sequence
\begin{equation}
0\rightarrow \mathcal{F}'_{y}\rightarrow E_{y}\rightarrow \mathcal{F}_{y}\rightarrow 0
\label{eq:4.26}
\end{equation}
is still exact, which implies that $\mathcal{F}$ is $B-$flat(cf Lemma 2.1.4 of \cite{HuybrechtsLehn}). The sequence (\ref{eq:4.26}) already implies $\text{deg}f_{*}E=\text{deg}\mathcal{F}'_{y}\leq 0$ since $E_{y}$ is stable of degree 1. Thus $\mathcal{F}$ can not be locally free since
$$6c_{2}(\mathcal{F})=\Delta(E)-12\text{deg}f_{*}E+8>0.$$
Then there is at least a $y_{0}\in B$ such that $\mathcal{F}_{y_{0}}$ has torsion, otherwise $\mathcal{F}$ is locally free (cf Lemma 1.27 of \cite{Simpson}). The stability of $E_{y_{0}}$ implies that $\mathcal{F}_{y_{0}}/torsion$ has degree at least 1. Thus $\text{deg}\mathcal{F}_{y_{0}}\geq 2$ and
$$\text{deg}f_{*}E=\text{deg}\mathcal{F}'_{y_{0}}\leq -1.$$
Which means $2\sum_{i=1}^{s}\text{rk}G_{i}\geq 5+3\text{dim}H^{0}(G_{s})$. When $\text{rk}G_{s}=1$, if $\mu(G_{s})<0$, then $\Delta(E)\geq \text{Min}_{2(s_{1}+2s_{2}+1)\geq 5}\{4s_{1}+2s_{2}+6(\frac{2}{3}-(-1))\}\geq 12$. If $\mu(G_{s})=0$, tensoring $E$  with $\pi^{*}G_{s}^{-1}$, we can assume that $H^{0}(G_{s})\neq0$, then $\Delta(E)\geq \text{Min}_{2(s_{1}+2s_{2}+1)\geq 8}\{4s_{1}+2s_{2}+6\cdot\frac{2}{3}\}\geq 8$. Now we consider the case $\text{rk}G_{s}=2$. If $\mu(G_{s})<0$, then $\Delta(E)\geq \text{Min}_{2(s_{1}+2s_{2}+2)\geq 5}\{4s_{1}+2s_{2}+6(\frac{2}{3}-(-\frac{1}{2}))\cdot 2\}\geq 16$. If $\mu(G_{s})=0$, tensoring $E$ with a suitable line bundle $\pi^{*}L^{-1}$, we can assume that $H^{0}(G_{s})\neq0$(cf. Theorem 5 of \cite{Atiyah}), then $\Delta(E)\geq \text{Min}_{2(s_{1}+2s_{2}+2)\geq 8} \{4s_{1}+2s_{2}+6\cdot \frac{2}{3}\cdot 2\}\geq 8$. If $\mu(G_{2})=\frac{1}{2}$, we can prove that $\Delta(E)=\sum_{i=1}^{s-1}6(\frac{2}{3}-\mu(G_{i}))\text{rk}G_{i}+2\geq 8$ as following:

If $s_{1}>1$, it's easy to see that $\Delta(E)\geq 4s_{1}+2s_{2}+2\geq 10$. If $s_{1}=1$ then $s_{2}\geq1$ since $2(s_{1}+2s_{2}+2)\geq 2-3\text{deg}f_{*}E+3\text{dim}H^{0}(G_{s})=5-3\text{deg}f_{*}E\geq8$, thus $\Delta(E)\geq4s_{1}+2s_{1}+2\geq 4\times1+2\times 1+2=8$. Now we assume that $s_{1}=0$, then we must have either $s_{2}>2$ or $\exists i$ such that $\mu(G_{i})\leq 0$. (In fact, we note that $s_{2}\geq1$ since $2(s_{1}+2s_{2}+2)\geq 2-3\text{deg}f_{*}E+3\text{dim}H^{0}(G_{s})=5-3\text{deg}f_{*}E\geq8$, if $s_{2}=1$ and $\mu(G_{1})=\mu(G_{2})=\frac{1}{2}$, then $s=2$ and, by the following Lemma, we have $\text{deg}f_{*}E^{(2)}=\text{deg}f_{*}E-2=-3$ since $2(s_{1}+2s_{2}+2)\geq5-3\text{deg}f_{*}E\geq8$, which contradicts to equation (\ref{eq:4.25"}).
If $s_{2}=2$ and $\mu(G_{1})=\mu(G_{2})=\mu_{G_{3}}=\frac{1}{2}$, then $s=3$ and $\text{deg}E^{(3)}_{y}=1-\sum_{i=1}^{3}2=-5$ by (\ref{eq:4.25.2}), which contradicts equation (\ref{eq:4.25.3}).)
 If $s_{2}>2$, we also have $\Delta(E)\geq2s_{2}+2\geq8$. If $\exists i$ such that $\mu(G_{i})\leq 0$, then $\Delta(E)=\sum_{i=1}^{s-1}6(\frac{2}{3}-\mu(G_{i}))\text{rk}G_{i}+2\geq 6\times\frac{2}{3}\times 2+2=10$.

If $\phi: B\rightarrow M$ passes through a generic point, i.e., a (1,10)-stable point, we claim that $\text{deg}f_{*}E\leq -8$. To prove the claim, we assume that $\text{deg}f_{*}E=-m$ (where $m\in \{1, 2, 3, 4, 5, 6, 7\}$), we will show that $\phi(B)$ lies in a proper closed subset. Note that $\mathcal{F}_{y}$ must be locally free of rank 1 and degree $1+m$ for generic $y\in B$( if $\mathcal{F}_{y}$ has nontrivial torsion, then $E_{y}$ has a quotient line bundle of degree at most $m$, which is impossible since $E_{y}$ is (1, 10)-stable for generic $y\in B$). Thus $E_{y}$ satisfies $0\rightarrow V_{1}\rightarrow E_{y}\rightarrow V_{2}\rightarrow 0$, where $V_{1}, V_{2}$ are vector bundles on $C$ of ranks 2, 1 and degrees $-m, 1+m$ respectively such that $\text{det}V_{1}\times V_{2}\cong \mathcal{L}$. The locus of such bundles has dimension at most $4(g-1)+1+g+h^{1}(V_{2}^{-1}\otimes V_{1})-1-g=6(g-1)+3m+2<\text{dim}M$ when $g> 12$. Thus $\text{deg}f_{*}E\leq -8$ and $2\sum_{i=1}^{s}\text{rk}(G_{i})\geq 2-3\text{deg}f_{*}E+3\text{dim}H^{0}(G_{s})\geq 26+3\text{dim}H^{0}(G_{s})$. We consider the case that $\text{rk}(G_{s})=1$ at first, if $\mu(G_{s})<0$, then $\Delta(E)\geq \text{Min}_{2(s_{1}+2s_{2}+1)\geq 26}\{4s_{1}+2s_{2}+6(\frac{2}{3}-(-1))\}\geq 22$. If $\mu(G_{s}=0)$, tensoring $E$ with $\pi^{*}G_{s}^{-1}$, we can assume that $H^{0}(G_{s})\neq0$, then $\Delta(E)\geq\text{Min}_{2(s_{1}+2s_{2}+1)\geq29}\{4s_{1}+2s_{2}+6\times \frac{2}{3}\}\geq 18$. Now we consider the case that $\text{rk}(G_{s})=2$. If $\mu(G_{s})<0$, then $\Delta(E)\geq \text{Min}_{2(s_{1}+2s_{2}+2)\geq 26}\{4s_{1}+2s_{2}+6(\frac{2}{3}-(-\frac{1}{2}))\cdot 2\} \geq22$. If $\mu(G_{s})=0$, tensoring $E$ with a suitable line bundle $\pi^{*}L^{-1}$, we can assume that $H^{0}(G_{s})\neq0$(cf. Theorem 5 of \cite{Atiyah}), then $\Delta(E)\geq\text{Min}_{2(s_{1}+2s_{2}+2)\geq29}\{4s_{1}+2s_{2}+6\times \frac{2}{3}\times 2\}\geq 22$. If $\mu(G_{s})=\frac{1}{2}$, if $s_{1}\geq1$, then $\Delta(E)\geq 2s_{2}+6(\frac{2}{3}-\frac{1}{2})\times 2\geq 18$ since $2(s_{1}+2s_{2}+2)\geq 26+3=29$. If $s_{1}=0$, considering the inequality (\ref{eq:4.25'}), then $\Delta(E)=\sum_{i=1}^{s-1}+6(\frac{2}{3}-\frac{1}{2})\times 2<18$ if and only if $s_{2}=7$, $\mu(G_{1})=\cdots=\mu(G_{7})=\mu(G_{8})=\frac{1}{2}$ and $-9\leq\text{deg}f_{*}E\leq -8$. By the following Lemma, we have $\text{deg}f_{*}E^{(s)}=\text{deg}f_{*}E-8$. By equation (\ref{eq:4.25"}), $\text{deg}f_{*}E=-8$ and $\text{deg}L=-8$. But, on the other hand, by equations (\ref{eq:4.25.2}) and (\ref{eq:4.25.3}), we have $\text{deg}L=-5$. The contradiction implies that $\Delta(E)\geq 18$.
\end{proof}

\begin{Lemma}
Keep the notations as Proposition \ref{prop:4.5}. If for any $i\in \{1, \cdots, s\}$, $\text{rk}G_{i}=2$ and $\mu(G_{i})=\frac{1}{2}$, then we have $\text{deg}f_{*}E^{(s)}=\text{deg}f_{*}E-s$.
\end{Lemma}

\begin{proof}
Since $G_{i}$ is an indecomposable vector bundle on $X_{t_{i}}=\{t_{i}\}\times B\cong B$ of rank 2 and degree 1, then by Lemma 15 in \cite{Atiyah} and Riemann-Roch Theorem, we have
$$\text{dim}H^{0}(G_{i})=1 \quad\text{ and }\quad H^{1}(G_{i})=0.$$
By the definition of $E^{(i)}$, we have
$$0\rightarrow E^{(i)}\rightarrow E^{(i-1)}\rightarrow _{X_{t_{i}}}G_{i}\rightarrow 0.$$
Take direction of above sequence, we have
\begin{equation}
0\rightarrow f_{*}E^{(i)}\rightarrow f_{*}E^{(i-1)}\rightarrow H^{0}(G_{i})\rightarrow R^{1}f_{*}E^{(i)}\rightarrow R^{1}f_{*}E^{(i-1)}\rightarrow 0.
\label{eq:4.26'}
\end{equation}
If $R^{1}f_{*}E^{(i)}=0$, then by (\ref{eq:4.26'}), we have
\begin{equation}
\text{deg}f_{*}E^{(i)}=\text{deg}f_{*}E^{(i-1)}-1 \quad\text{ and }\quad
R^{1}f_{*}E^{(i-1)}=0 .
\label{eq:4.26"}
\end{equation}
We note that $R^{1}f_{*}E^{(s)}=0$, then $R^{1}f_{*}E^{(s-1)}=\cdots=R^{1}f_{*}E^{(1)}=R^{1}f_{*}E=0$ and
$$\text{deg}f_{*}E^{(s)}=\text{deg}f_{*}E^{(s-1)}-1=\cdots=\text{deg}f_{*}E-s.$$
\end{proof}

Before consider the case that $E$ is semi-stable of degree 0 on the generic fiber of $f: X\rightarrow C$, we note that:

(1) For any vector bundle $E$, $\Delta(E)=\Delta(E^{*})$ where $E^{*}$ is the dual of $E$.

(2)If $\phi: B\rightarrow M=SU_{C}(3, \mathcal{L})$ is defined by a vector bundle $E$ on $C\times B$, let $\phi_{E^{*}}: B\rightarrow M^{*}=SU_{C}(3, \mathcal{L}^{-1})$ be the morphism defined by $E^{*}$. Then $\phi: B\rightarrow M=SU_{C}(3, \mathcal{L})$ can factors as the composition of $\phi_{E^{*}}$ with the natural isomorphism $M^{*}\cong M, W\mapsto W^{*}$.

(3)$E^{*}_{t}=E^{*}|_{X_{t}}$ is semi-stable on a fiber $X_{t}=f^{-1}(t)$ if and only if $E_{t}$ is semi-stable.
\vspace{0.1cm}

Now we consider the case that $E$ is semi-stable of degree 0 on the generic fiber of $f: X\rightarrow C$. If $E$ is semi-stable on every fiber of $f: X\rightarrow C$, then $E$ induces a non-trivial morphism
$$\varphi_{E}: C\rightarrow \mathbb{P}^{2}$$
(cf. \cite{FriedmanMorganWitten}) such that $\varphi_{E}^{*}\mathcal{O}_{\mathbb{P}^{2}}(1)=(\text{det}f_{!}E)^{-1}$, which has degree $c_{2}(E)$ by Grothendieck-Riemann-Roch theorem. Thus
\begin{equation}
\Delta(E)=6c_{2}(E)=6\text{deg}\varphi_{E}\geq12.
\label{eq:4.27}
\end{equation}

If there is a $t_{0}\in C$ such that $E_{t_{0}}=E|_{X_{t_{0}}}$ is not semi-stable on $X_{t_{0}}=f^{-1}(t_{0})$, let $E_{t_{0}}\rightarrow G\rightarrow 0$ be the indecomposable  quotient bundle of minimal slop $\mu$. If $\text{rk}G=1$, $G$ is a line bundle of degree $\mu=\text{deg}G$ and we will denote $\mathcal{O}(\mu):=G$, and $E'=\text{ker}(E\rightarrow _{X_{t_{0}}}\mathcal{O}(\mu)\rightarrow 0)$.
 If $\text{rk}G=2$, let $\mu^{*}:=2\mu=\text{deg}G$. Then $\text{ker}(E_{t_{0}}\rightarrow G\rightarrow 0)$ is a line bundle of degree $-\mu^{*}$. Take dual, $E_{t_{0}}^{*}$ is not semi-stable and $E_{t_{0}}^{*}\rightarrow \mathcal{O}(\mu^{*})\rightarrow0$ is a minimal quotient line bundle of degree $\mu^{*}$. Let $E^{*'}=\text{ker}(E^{*}\rightarrow _{X_{t_{0}}}\mathcal{O}(\mu^{*})\rightarrow0)$.

\begin{Lemma}
If $\text{rk}G=1$ and $\Delta(E')=0$, then there is a semi-stable vector bundle $V$ on $C$ and a line bundle $L$ of degree 0 on $B$ such that
$$E'=f^{*}V\otimes \pi^{*}L.$$
\label{lm:4.1}
\end{Lemma}

\begin{proof}
By the definition, $\{E'_{t}:=E'|_{\{t\}\times B}\}_{t\in C}$ and $\{E'_{y}:=E'|_{C\times\{y\}}\}_{y\in B}$ are families of semi-stable bundles of degree 0. Apply Theorem 2.4 to $f: X\rightarrow C$ (resp. $\pi: X\rightarrow B$), then $\Delta(E')=0$ implies that $\{E'_{y}\}_{y\in B}$ (resp. $\{E'_{t}\}_{t\in C}$) are isomorphic to each other. Let $L$ be the line bundle of degree 0 on $B$ such that $H^{0}(E'_{t}\otimes L^{-1})(\forall t\in C)$ have maximal dimension. By tensoring $E$ (thus $E'$) with $\pi^{*}L^{-1}$, we can assume that $H^{0}(E'_{t})\neq0 (\forall t\in C)$, which have dimension at most 3 since $E'_{t}$ is semi-stable of degree 0. If $H^{0}(E'_{t})$ has dimension 3, then $E'=f^{*}(f_{*}E')$ and we are done. If $H^{0}(E'_{t})$ has dimension 1 or 2, we will show contradictions.

 By the definition of $E'$, we have an exact sequence
\begin{equation}
0\rightarrow E'\rightarrow E\rightarrow _{X_{t_{0}}}\mathcal{O}(\mu)\rightarrow 0,
\label{eq:4.28}
\end{equation}
where $\mathcal{O}(\mu)$ is a line bundle of degree $\mu<0$ on $B$. Then
$$V_{1}:=f_{*}E=f_{*}E'$$
is a vector bundle on $C$.

If $H^{0}(E'_{t})$ has dimension 1, then $V_{1}:=f_{*}E'=f_{*}E$ is a line bundle and we have exact sequence
\begin{equation}
0\rightarrow f^{*}V_{1}\rightarrow E'\rightarrow \mathcal{F}'\rightarrow 0
\label{eq:4.37}
\end{equation}
for a rank 2 vector bundle $\mathcal{F}'$ on $C\times B$ and $\Delta(\mathcal{F}')=0$. All the bundles $\{\mathcal{F}'_{t}=\mathcal{F}'|_{\{t\}\times B}\}_{t\in C}$ are semi-stable of degree 0 and are isomorphic to each other since all the bundles $\{E'_{t}\}_{t\in C}$ are semi-stable of degree 0 and are isomorphic to each other. Let $L'$ be a line bundle of degree 0 on $B$ such that $H^{0}(\mathcal{F}'_{t}\otimes L')\neq 0 (\forall t\in C)$, which must have dimension 1. To see it, tensor (\ref{eq:4.37}) with $\pi^{*}L'$, we have an exact sequence
\begin{equation}
0\rightarrow f^{*}V_{1}\otimes \pi^{*}L'\rightarrow E'\otimes \pi^{*}L'\rightarrow \mathcal{F}'\otimes \pi^{*}L'\rightarrow 0.
\label{eq:4.38}
\end{equation}
Restrict (\ref{eq:4.38}) to $X_{t}=\{t\}\times B\cong B$, we have
$$0\rightarrow L'\rightarrow E'_{t}\otimes L'\rightarrow\mathcal{F}'_{t}\otimes L'\rightarrow 0,$$
and then we have
$$0\rightarrow H^{0}(L')\rightarrow H^{0}(E'_{t}\otimes L')\rightarrow H^{0}(\mathcal{F}'_{t}\otimes L')\rightarrow H^{1}(L')\rightarrow \cdots,$$
which implies $h^{0}(\mathcal{F}'_{t}\otimes L')\leq h^{0}(E'_{t}\otimes L')\leq 1$ by the choice of $L$. Then $V_{2}=f_{*}(\mathcal{F}'\otimes \pi^{*}L')$ is a line bundle on $C$, and we have exact sequence
\begin{equation}
0\rightarrow f^{*}V_{2}\rightarrow \mathcal{F}'\otimes \pi^{*}L'\rightarrow f^{*}V_{3}\otimes \pi^{*}L"\rightarrow 0
\label{eq:4.39}
\end{equation}
for a line bundle $V_{3}$ on $C$ and a degree 0 line bundle $L"$ on $B$.

 Now we note that $\text{deg}V_{1}+\text{deg}V_{2}\leq -1$. To see it, we consider the exact sequence
 \begin{equation}
 0\rightarrow f^{*}V_{1}\rightarrow E\rightarrow \mathcal{F}\rightarrow 0
 \label{eq:4.35'}
 \end{equation}
 where $\mathcal{F}|_{f^{-1}(C\setminus \{t_{0}\})}$ is locally free and $\mathcal{F}$ satisfies
 \begin{equation}
 0\rightarrow \mathcal{F}'\rightarrow \mathcal{F}\rightarrow _{X_{t_{0}}}\mathcal{O}(\mu)\rightarrow 0.
 \label{eq:4.40}
 \end{equation}
 Thus $f_{*}(\mathcal{F}\otimes \pi^{*}L')=f_{*}(\mathcal{F}'\otimes \pi^{*}L')=V_{2}$ since $\mu<0$. Then we have an exact sequence
 \begin{equation}
 0\rightarrow f^{*}V_{2}\rightarrow \mathcal{F}\otimes \pi^{*}L'\rightarrow \mathcal{G}\rightarrow 0
 \label{eq:4.41}
 \end{equation}
 where $\mathcal{G}|_{f^{-1}(C\setminus \{t_{0}\})}$ is locally free of rank 1 by (\ref{eq:4.39}). But $\mathcal{G}$ is not locally free (otherwise $c_{2}(E)=c_{2}(E\otimes\pi^{*}L')
 =c_{1}(f^{*}V_{1}\otimes\pi^{*}L')(c_{1}(E\otimes\pi^{*}L')-c_{1}(f^{*}V_{1}\otimes\pi^{*}L'))
 +c_{2}(\mathcal{F}\otimes\pi^{*}L')
 =c_{1}(f^{*}V_{2})(c_{1}(E\otimes\pi^{*}L')-c_{1}(f^{*}V_{1}\otimes\pi^{*}L')-c_{1}(f^{*}V_{2}))=0$), and for any $y\in B$, the restrictions of (\ref{eq:4.35'}) and (\ref{eq:4.41}) to $X_{y}=\pi^{-1}(y)$
 $$0\rightarrow V_{1}\rightarrow E_{y}\rightarrow \mathcal{F}_{y}\rightarrow 0
 \text{ and } 0\rightarrow V_{2}\rightarrow \mathcal{F}_{y}\rightarrow \mathcal{G}_{y}\rightarrow 0$$
 are still exact, which means $\mathcal{F}$ is $B$-flat and then $\mathcal{G}$ is $B$-flat(cf. Lemma 2.1.4 of \cite{HuybrechtsLehn}). Thus, by Lemma 1.27 of \cite{Simpson}, there is a $y_{0}\in B$ such that $\mathcal{G}_{y_{0}}$ has torsion $\tau\neq0$ since $\mathcal{G}$ is not locally free. Then, since $E_{y_{0}}$ is stable of degree 1,
 $$\text{deg}\mathcal{G}_{y_{0}}\geq 1+\text{deg}\frac{\mathcal{G}_{y_{0}}}{\tau}>1+\mu(E_{y_{0}})=\frac{4}{3}$$
 which implies $\text{deg}V_{1}+\text{deg}V_{2}=\text{deg}E_{y_{0}}-\text{deg}\mathcal{G}_{y_{0}}\leq-1$.

By the sequences (\ref{eq:4.38}) and (\ref{eq:4.39}), $f^{*}V_{3}\otimes \pi^{*}L"$ is a quotient line bundle of $E'\otimes \pi^{*}L'$. Let $F:= \text{ker}(E'\otimes \pi^{*}L'\rightarrow f^{*}V_{3}\otimes \pi^{*}L"\rightarrow 0)$, then there is an induced morphism $\lambda: F\rightarrow f^{*}V_{2}$ satisfying the diagram
\[\begin{CD}
0@>>>F@>>>E'\otimes \pi^{*}L'@>>>f^{*}V_{3}\otimes \pi^{*}L"@>>>0\\
@.@V\lambda VV@VVV@|\\
0@>>>f^{*}V_{2}@>>>\mathcal{F}'\otimes \pi^{*}L'@>>>f^{*}V_{3}\otimes \pi^{*}L"@>>>0.
\end{CD} \]
By the snake lemma, $\lambda$ is surjective and $\text{ker}\lambda= f^{*}V_{1}\otimes \pi^{*}L'$. Then $F$ fits an exact sequence
\begin{equation}
0\rightarrow f^{*}V_{1}\otimes \pi^{*}L'\rightarrow F\rightarrow f^{*}V_{2}\rightarrow 0,
\label{eq:4.42}
\end{equation}
which is determined by a class in $H^{1}(X,f^{*}(V_{2}^{-1}\otimes V_{1})\otimes \pi^{*}L')$.

If $L'\neq \mathcal{O}_{B}$, then $R^{i}f_{*}(f^{*}(V_{2}^{-1}\otimes V_{1})\otimes \pi^{*}L')=V_{2}^{-1}\otimes V_{1}\otimes H^{i}(L')=0 (i=0, 1)$, which implies $H^{1}(X,f^{*}(V_{2}^{-1}\otimes V_{1})\otimes \pi^{*}L')=0$ and (\ref{eq:4.42})is split. Thus there is a section $f^{*}V_{2}\rightarrow F$ of $\lambda$, and we can consider $f^{*}V_{2}$ as a sub line bundle of $E'\otimes \pi^{*}L'$ by the morphism $f^{*}V_{2}\rightarrow F\rightarrow E'\otimes \pi^{*}L'$. Then $\text{deg}V_{2}\leq0$ since $E'_{y}$ is semi-stable of degree 0 for any $y\in B$. If $L"\neq\mathcal{O}_{B}$, then (\ref{eq:4.39}) is also split, then we have an exact sequence of inverse direction
$$0\rightarrow f^{*}V_{3}\otimes \pi^{*}L"\rightarrow \mathcal{F}'\otimes \pi^{*}L'\rightarrow f^{*}V_{2}\rightarrow 0.$$
Hence $\text{deg}V_{2}=0$. Now we let $F'=\text{ker}(E'\otimes \pi^{*}L'\rightarrow f^{*}V_{2}\rightarrow 0)$, then there is an induced morphism $F'\rightarrow f^{*}V_{3}\otimes \pi^{*}L"$ and $F'$ satisfies an exact sequence
\begin{equation}
0\rightarrow f^{*}V_{1}\otimes \pi^{*}L'\rightarrow F'\rightarrow f^{*}V_{3}\otimes \pi^{*}L"\rightarrow 0,
\label{eq:4.43}
\end{equation}
which is determined by a class in $H^{1}(X,f^{*}(V_{3}^{-1}\otimes V_{1})\otimes \pi^{*}(L"^{-1}\otimes L'))$. If $L"\neq L'$, we can prove that (\ref{eq:4.43}) is split. Then $f^{*}V_{3}\otimes \pi^{*}L"$ is a sub line bundle of $F'$ and then $f^{*}V_{3}\otimes \pi^{*}L"$ is a sub line bundle of $E'\otimes \pi^{*}L'$. Thus $\text{deg}V_{3}\leq 0$ since $E'_{y}$ is semi-stable of degree 0 for any $y\in B$, which contradicts that $\text{deg}V_{1}+\text{deg}V_{2}\leq-1$. Hence $L"=L'$. Tensoring (\ref{eq:4.43}) with $\pi^{*}L'^{-1}$, we have an exact sequence
\begin{equation}
0\rightarrow f^{*}V_{1}\rightarrow F'\otimes \pi^{*}L'^{-1}\rightarrow f^{*}V_{3}\rightarrow 0,
\label{eq:4.44}
\end{equation}
which is determined by a class in $H^{1}(X,f^{*}(V_{3}^{-1}\otimes V_{1}))$. However, note $R^{i}f_{*}(f^{*}(V_{3}^{-1}\otimes V_{1}))=V_{3}^{-1}\otimes V_{1} (i=0, 1)$ and $H^{0}(C, V_{3}^{-1}\otimes V_{1})=0$ since $\text{deg}V_{2}=0$, by Leray spectral sequence, we have
$$H^{1}(C,V_{3}^{-1}\otimes V_{1})\cong H^{1}(X, f^{*}(V_{3}^{-1}\otimes V_{1})).$$
Hence there exists an extension $0\rightarrow V_{1}\rightarrow V'\rightarrow V_{3}\rightarrow 0$ on $C$ such that $F'\otimes \pi^{*}L'^{-1}=f^{*}V'$. Thus $h^{0}(E'_{t})\geq h^{0}((F'\otimes \pi^{*}L'^{-1})_{t})=2$, which contradicts the assumption $h^{0}(E'_{t})=1$.
Thus $L"$ has to be $\mathcal{O}_{B}$ and (\ref{eq:4.39}) has to be
\begin{equation}
0\rightarrow f^{*}V_{2}\rightarrow \mathcal{F}'\otimes \pi^{*}L'\rightarrow f^{*}V_{3}\rightarrow 0
\label{eq:4.45}
\end{equation}
which is determined by a class in $H^{1}(X, f^{*}(V_{3}^{-1}\otimes V_{2}))$. However, note that $R^{i}f_{*}(f^{*}(V_{3}^{-1}\otimes V_{2}))=V_{3}^{-1}\otimes V_{2} (i=0, 1)$ and $H^{0}(C,V_{3}^{-1}\otimes V_{2})$=0 since $\text{deg}V_{2}\leq 0$, by Leray spectral sequence, we have
$$H^{1}(C,V_{3}^{-1}\otimes V_{2})\cong H^{1}(X, f^{*}(V_{3}^{-1}\otimes V_{2})).$$
Hence there exists an extension $0\rightarrow V_{2}\rightarrow W'\rightarrow V_{3}\rightarrow 0$ on $C$ such that $\mathcal{F}'\otimes \pi^{*}L'=f^{*}W'$. Thus $h^{0}(E'_{t}\otimes L')\geq h^{0}((\mathcal{F}'\otimes \pi^{*}L')_{t})=2$, which contradict the choice of $L$ and that $h^{0}(E'_{t})=1$.

We have shown that $L'$ has to be $\mathcal{O}_{B}$ and (\ref{eq:4.42}) has to be
\begin{equation}
0\rightarrow f^{*}V_{1}\rightarrow F\rightarrow f^{*}V_{2}\rightarrow 0,
\label{eq:4.46}
\end{equation}
which is determined by a class in $H^{1}(X,f^{*}(V_{2}^{-1}\otimes V_{1}))$. If $\text{deg}V_{2}>\text{deg}V_{1}$, we can see that
$$H^{1}(C, V_{2}^{-1}\otimes V_{1})\cong H^{1}(X,f^{*}(V_{2}^{-1}\otimes V_{1})).$$
Hence there exists an extension $0\rightarrow V_{1}\rightarrow W\rightarrow V_{2}\rightarrow 0$ on $C$ such that $F=f^{*}W$. Thus $h^{0}(E'_{t})\geq h^{0}(F_{t})=2$, which contradicts that $h^{0}(E'_{t})=1$. So $\text{deg}V_{2}\leq\text{deg}V_{1}$. Since $\text{deg}V_{1}+\text{deg}V_{2}\leq-1$ and $\text{deg}V_{1}\leq 0$, then $\text{deg}V_{2}\leq-1$ and $L"=\mathcal{O}_{B}$ (otherwise, the sequence (\ref{eq:4.39}) is split and then $\text{deg}V_{2}\geq 0$). Now (\ref{eq:4.39}) has to be (\ref{eq:4.45}), which is determined by a class in $H^{1}(X, f^{*}(V_{3}^{-1}\otimes V_{2}))\cong H^{1}(C,V_{3}^{-1}\otimes V_{2})$. which implies that $h^{0}(E'_{t}\otimes L')\geq h^{0}((\mathcal{F}'\otimes \pi^{*}L')_{t})=2$, which contradict the choice of $L$ and that $h^{0}(E'_{t})=1$.

If $H^{0}(E'_{t})$ has dimension 2, we can also show a contradiction similar as in Lemma 4.4 in \cite{Sun}.
\end{proof}

\begin{Proposition}
When $E$ is semi-stable of degree 0 on the generic fiber of $f: X\rightarrow C$, we have $\Delta(E)\geq 6$. If $C$ is not hyper-elliptic and $\phi: B\rightarrow M$ passes through generic points, assume that $E$ defines an essential elliptic curves, then $\Delta(E)\geq18$ when $g\geq 4$.
\label{prop:4.6}
\end{Proposition}

\begin{proof}
If $E$ is semi-stable on each fiber of $f: X\rightarrow C$, then $E$ induces a non-trivial morphism $\varphi_{E}: B\rightarrow \mathbb{P}^{2}$. By (\ref{eq:4.27}), $\Delta(E)\geq12$.

If there is a $t_{0}\in C$ such that $E_{t_{0}}$ is not semi-stable, then we have either (\ref{eq:4.28})
or
\begin{equation}
 0\rightarrow E^{*'}\rightarrow E^{*}\rightarrow _{X_{t_{0}}}\mathcal{O}(\mu^{*})\rightarrow0,
 \label{eq:4.29}
 \end{equation}
 where $\mathcal{O}(\mu^{*})$ is a line bundle of degree $\mu^{*}\leq-1$ on $B$.

 If we have (\ref{eq:4.28}). If $\Delta(E')\neq 0$, then $\Delta(E')>0$ by Theorem \ref{th:2.4}. On the other hand, $c_{1}(E')^{2}=0$ since $E'$ has degree 0 on the generic fiber of $f: X\rightarrow C$ and $\text{Pic}(C\times B)=\text{Pic}(C)\times \text{Pic}(B)$. Thus $\Delta(E')=6c_{2}(E')\geq6$, and by Lemma \ref{lm2.6}, we have $\Delta(E)=\Delta(E')-6\mu\geq 12$. If $\Delta(E')=0$, by Lemma \ref{lm:4.1}, we can assume that $E'=f^{*}V$, then the sequence (\ref{eq:4.28}) induces a nontrivial morphism $\varphi: B\rightarrow \mathbb{P}(V_{t_{0}}^{*})$ such that $\mathcal{O}(-\mu)=\varphi^{*}\mathcal{O}_{\mathbb{P}(V_{t_{0}}^{*})}(1)$. Thus $\Delta(E)=-6\mu\geq 12$.

 If we have (\ref{eq:4.29}), by Lemma \ref{lm2.6}, we have
  $\Delta(E)=\Delta(E^{*})=\Delta(E^{*'})-6\mu^{*}\geq 6$.

Now we assume that $C$ is not hyper-elliptic and $\phi: B\rightarrow M$ passes through generic points, i.e., $E_{y}|_{C\times\{y\}}$ is not only (1,1)-stable but also (1,2)-stable for generic $y\in B$. If $E$ is semi-stable on each fiber $X_{t}$, then $\Delta(E)\geq6\text{deg}\varphi_{E}\geq18$ by (\ref{eq:4.27}) since $C$ is not hyper-elliptic.

If there is a $t_{0}\in C$ such that $E_{t_{0}}$ is not semi-stable, then we have either (\ref{eq:4.28}) or (\ref{eq:4.29}).

We consider the case that (\ref{eq:4.28}) holds at first, for this case we have
$\Delta(E)=\Delta(E')-6\mu$. If $\Delta(E')=0$, then $E'=f^{*}V$ where $V$ is a (1,0)-stable by Lemma 3.3, then the sequence (\ref{eq:4.28}) induces a non-trivial morphism $\varphi: B\rightarrow \mathbb{P}(V_{t_{0}}^{*})$ such that $\mathcal{O}(-\mu)=\varphi^{*}\mathcal{O}_{\mathbb{P}(V_{t_{0}}^{*})}(1)$ and $\phi: B\rightarrow M$ factors through $\varphi: B\rightarrow \varphi(B)\subset \mathbb{P}(V_{t_{0}}^{*})$, which implies that the normalization of $\varphi(B)$ is an elliptic curve. Hence $-\mu\geq 3$ and $\Delta(E)\geq 18$.

Now we consider the case $\Delta(E')>0$. We claim that $\Delta(E')\geq12$, which implies $\Delta(E)\geq18$. If $E'$ is semi-stable on each fiber $X_{t}$, then $E'$ defines a non-trivial morphism $\varphi_{E'}: C\rightarrow \mathbb{P}^{2}$ such that $\varphi_{E'}^{*}\mathcal{O}_{\mathbb{P}^{2}}(1)=(\text{det}f_{!}E')^{-1}=c_{2}(E')$. Thus $\Delta(E')\geq12$.
If there is a $t'_{0}\in C$ such that $E'_{t_{0}}$ is not semi-stable, then we have either
\begin{equation}
0\rightarrow E"\rightarrow E'\rightarrow _{X_{t'_{0}}}\mathcal{O}(\mu')\rightarrow 0
\label{eq:4.30}
\end{equation}
where $E"_{y}=E"|_{C\times \{y\}}$ is stable of degree -1 for generic $y\in B$ since $E'_{y}$ is stable of degree 0, or
\begin{equation}
0\rightarrow E^{'*'}\rightarrow E^{'*}\rightarrow _{X_{t'_{0}}}\mathcal{O}(\mu^{'*})\rightarrow 0
\label{eq:4.31}
\end{equation}
where $E^{'*'}_{y}=E^{'*'}|_{C\times \{y\}}$ is stable of degree -1 since $E^{'*}_{y}=(E'_{y})^{*}$ is stable of degree 0. Suppose that (\ref{eq:4.30}) holds, if $\Delta(E")\neq0$, it's clear that $\Delta(E")=6c_{2}(E")\geq 6$ and $\Delta(E')=\Delta(E")-6\mu'\geq12$. If $\Delta(E")=0$, by Theorem2.4, there is a stable bundle $V'$ on $C$ such that $E"_{y}=V'$ for all $y\in B$. Then we can choose $E"=f^{*}V'$, the sequence (\ref{eq:4.30}) induces a non-trivial morphism $\varphi': B\rightarrow \mathbb{P}(V^{'*}_{t'_{0}})$ such that $\mathcal{O}(-\mu')=\varphi^{'*}\mathcal{O}_{\mathbb{P}(V^{'*}_{t'_{0}})}(1)$. Thus $\Delta(E')=-6\mu'\geq12$. Now we suppose (\ref{eq:4.31}) holds. If $\Delta(E^{'*'})\neq0$, it's clear that $\Delta(E^{'*'})=6c_{2}(E^{'*'})\geq 6$ and $\Delta(E')=\Delta( E^{'*})=\Delta(E^{'*'})-6\mu^{'*}\geq12$. If $\Delta(E^{'*'})=0$, by Theorem \ref{th:2.4}, there is a stable bundle $W'$ on $C$ such that $E^{'*'}_{y}=W'$ for all $y\in B$. Then we can choose $E^{'*'}=f^{*}W'$, the sequence (\ref{eq:4.31}) induces a non-trivial morphism $\psi': B\rightarrow \mathbb{P}(W^{'*}_{t'_{0}})$ such that $\mathcal{O}(-\mu^{'*})=\psi^{'*}\mathcal{O}_{\mathbb{P}(W^{'*}_{t'_{0}})}(1)$. Thus $-\mu^{'*}\geq 2$ and $\Delta(E')=\Delta(E^{'*})=-6\mu^{'*}\geq12$.

For the case that (\ref{eq:4.29}), we have $\Delta(E)=\Delta(E^{*})=\Delta(E^{*'})-6\mu^{*}$. If $\Delta(E^{*'})=0$, then $E^{*'}=f^{-1}W$ where $W$ is (2,0)-stable of degree -2 by Remark3.1(ii) and Lemma 3.3, then the sequence (\ref{eq:4.29}) induces a non-trivial morphism $\psi: B\rightarrow \mathbb{P}(W^{*}_{t_{0}})$ such that $\mathcal{O}(-\mu^{*})=\psi^{*}\mathcal{O}_{\mathbb{P}(W^{*}_{t_{0}})}(1)$ and $\phi: B\rightarrow M$ factors through $\psi: B\rightarrow \psi(B)\subset\mathbb{P}(W^{*}_{t_{0}})$ by $\mathbb{P}(W^{*}_{t_{0}})\rightarrow M^{*}\cong M$. Which implies that the normalization of $\psi(B)$ is an elliptic curve. Hence $-\mu^{*}\geq 3$ and $\Delta(E)=\Delta(E^{*})\geq 18$.
For the case $\Delta(E^{*'})\neq0$, we can prove that $\Delta(E^{*'})\geq12$ similarly as to prove that $\Delta(E')\geq12$, and hence $\Delta(E)=\Delta(E^{*})\geq 18$.
\end{proof}

From the Example 3.6 of \cite{Sun} and Proposition \ref{prop:3.1}, we can see the existence of essential elliptic curves of degree $6(r,d)$(which is 6 in our case). By Propositions \ref{prop:4.1}, \ref{prop:4.2}, \ref{prop:4.3}, \ref{prop:4.4}, \ref{prop:4.5}, \ref{prop:4.6}, we have

\begin{Theorem}\label{th:4.7}
Let $M=SU_{C}(3, \mathcal{L})$ be the moduli space of rank 3 stable bundles on $C$ with fixed determinant of degree 1. Then, when $C$ is generic, any essential elliptic curve $\phi: B\rightarrow M$ has degree
$$\text{deg}\phi^{*}(-K_{M})\geq 6$$
and $\text{deg}\phi^{*}(-K_{M})=6$ if and only if $\phi$ satisfies one of the following conditions:

(1)it factors through
\[\begin{CD}
\phi: B@>\psi>>q^{-1}(\xi)=\mathbb{P}(H^{1}(V_{2}^{*}\otimes V_{1}))@>\Phi_{\xi}>>M
\end{CD} \]
for some $\xi=(V_{1}, V_{2})\in J_{C}\times U_{C}(2,1)$ such that $\psi^{*}\mathcal{O}_{\mathbb{P}(H^{1}(V_{2}^{*}\otimes V_{1}))}(1)$ has degree 3.

(2)it factors through
\[\begin{CD}
\phi: B@>\psi>>\mathcal{P}@>\Phi_{\xi}>>M
\end{CD} \]
but it is not in any fiber of $q: \mathcal{P}\rightarrow \mathcal{R}_{\mathcal{L}}\hookrightarrow J_{C}\times U_{C}(2,1)$, and $q$ induces a morphism $q_{2}: B\rightarrow \mathbb{P}(H^{1}(L_{3}^{-1}\otimes L_{2}))$ for some $(L_{2},L_{3})\in J_{C}\times J_{C}^{1}$ such that $q_{2}^{*}\mathcal{O}_{\mathbb{P}(H^{1}(L_{3}^{-1}\otimes L_{2}))}(1)$ has degree 2 and $\psi^{*}\mathcal{O}_{\mathcal{P}}(1)$ has degree 1.

(3)it's defined by a vector bundle $E$ on $C\times B$, which is semi-stable of degree 0 at generic fiber of $f: X\rightarrow C$, there exists only one point $t_{0}\in C$ such that $E_{t_{0}}$ is not semi-stable and the minimal slop indecomposable quotient bundle $G$ of rank 2 and $\text{deg}G=-1$.
\end{Theorem}

\begin{proof}
By Propositions \ref{prop:4.1}, \ref{prop:4.2}, \ref{prop:4.3}, \ref{prop:4.4}, \ref{prop:4.5}, \ref{prop:4.6}, we have $\Delta(E)\geq6$.
Let $0=E_{0}\subset E_{1}\subset\cdots\subset E_{n}=E$ be the relative Harder-Narasimhan filtration of $E$ over $C$. The possible case $\Delta(E)=6$ only set up only in following three cases:

In Proposition \ref{prop:4.3} when $\text{deg}E_{1}=0$, $\Delta(F_{2})=0$ and $F_{2}$ is semistable of even degree $2\mu_{2}$ at the generic fiber of $f: X\rightarrow C$. The condition $\text{deg}E_{1}=0$ implies there is a line bundle $V_{1}$ of degree 0 on $C$ and a line bundle $\mathcal{O}(\mu_{1})$ of degree $\mu_{1}$ on $B$ such that $E_{1}=f^{*}V_{1}\otimes \pi^{*}\mathcal{O}(\mu_{1})$. Since $E$ is stable of degree 1 at every fiber of $\pi: X\rightarrow B$, $F_{2}$ is also stable of degree 1 at every fiber of $\pi: X\rightarrow B$. Applying Theorem \ref{th:2.4} to $\pi: X\rightarrow B$, there is a stable bundle $V_{2}$ of degree 1 on $C$ and a line bundle $\mathcal{O}(\mu_{2})$ of degree $\mu_{2}$ on $B$ such that $F_{2}=f^{*}V_{2}\otimes \pi^{*}\mathcal{O}(\mu_{2})$. These imply that $E\otimes \pi^{*}\mathcal{O}(-\mu_{2})$ fits an exact sequence
$$0\rightarrow f^{*}V_{1}\otimes\pi^{*}\mathcal{O}(\mu_{1}-\mu_{2})\rightarrow E\otimes \pi^{*}\mathcal{O}(-\mu_{2}) \rightarrow f^{*}V_{2}\rightarrow 0,$$
which defines a morphism $\phi: B\rightarrow \mathbb{P}(H^{1}(V_{2}^{*}\otimes V_{1}))$ such that $\psi^{*}\mathcal{O}_{\mathbb{P}(H^{1}(V_{2}^{*}\otimes V_{1}))}(1)$ is of degree $\mu_{1}-\mu_{2}$. Then $\Delta(E)=6$ and (\ref{eq:4.9}) imply $\mu_{1}-\mu_{2}=3$.

In Proposition \ref{prop:4.2} when $c_{2}(F_{2})=0, \Delta(E)=0$ and $E_{1}$ is semi-stable of odd degree $2\mu_{1}$ at generic fiber of $f: X\rightarrow C$. Tensoring $E$ with $\pi^{*}\mathcal{O}(m)$ for a degree $m$ line bundle $\mathcal{O}(m)$ on $B$ such that $\text{deg}(E_{1}\otimes\pi^{*}\mathcal{O}(m))=2\mu_{1}+2m=1$. Applying Theorem \ref{th:2.4} to
$f: X\rightarrow C$, $\Delta(E)=0$ implies all the bundles$\{E_{1t}\otimes \mathcal{O}(m)\}_{t\in C}$ are semistable of degree 1 and $s-$equivalent each other, and then stable and isomorphic to each other. Thus there is a stable bundle $V$ of degree 1 on $B$ and a line bundle $L_{1}$ on $C$ such that $E_{1}\otimes\pi^{*}\mathcal{O}(m)=\pi^{*}V\otimes f^{*}L_{1}$. Then we have $\text{deg}E_{1}=2\text{deg}L_{1}$ and $f_{*}(E_{1}\otimes\pi^{*}\mathcal{O}(m))\cong L_{1}$.
Thus $E_{1}\otimes\pi^{*}\mathcal{O}(m)$ satisfies an exact sequence
$$0\rightarrow f^{*}L_{1}\rightarrow E_{1}\otimes\pi^{*}\mathcal{O}(m)\rightarrow f^{*}L_{2}\otimes \pi^{*}\mathcal{O}(1)\rightarrow 0,$$
for a line bundle $L_{2}$ on $C$ and a line bundle $\mathcal{O}(1)$ of degree 1 on $B$.
On the other hand, $c_{2}(F_{2})=0$ implies there is a line bundle $L_{3}$ on $C$ and a line bundle $\mathcal{O}(\mu_{2})$ of degree $\mu_{2}$ on $B$ such that
$F_{2}=f^{*}L_{3}\otimes \pi^{*}\mathcal{O}(\mu_{2})$. Then $E\otimes \pi^{*}\mathcal{O}(m)$ fits an exact sequence
$$0\rightarrow E_{1}\otimes\pi^{*}\mathcal{O}(m)\rightarrow E\otimes\pi^{*}\mathcal{O}(m)\rightarrow f^{*}L_{3}\otimes\pi^{*}\mathcal{O}(\mu_{2}+m)\rightarrow 0.$$
 Consider $f^{*}L_{1}\otimes \pi^{*}\mathcal{O}(-\mu_{2}-m)$ as a subline bundle of $E\otimes \pi^{*}\mathcal{O}(-\mu_{2})$ and let $E':=\frac{E\otimes \pi^{*}\mathcal{O}(-\mu_{2})}{f^{*}L_{1}\otimes \pi^{*}\mathcal{O}(-\mu_{2}-m)}$, then there is an induced homomorphism $\eta: f^{*}L_{2}\otimes\pi^{*}\mathcal{O}(1-\mu_{2}-m)\rightarrow E'$. By the snake lemma, $\eta$ is injective and $E'$ fits exact sequences
\begin{equation}
0\rightarrow f^{*}L_{1}\otimes \pi^{*}\mathcal{O}(-\mu_{2}-m)\rightarrow E\otimes \pi^{*}\mathcal{O}(-\mu_{2})\rightarrow E'\rightarrow 0
\label{eq:4.32}
\end{equation}
and
\begin{equation}
0\rightarrow f^{*}L_{2}\otimes\pi^{*}\mathcal{O}(1-\mu_{2}-m)\rightarrow E'\rightarrow f^{*}L_{3}\rightarrow 0.
\label{eq:4.33}
\end{equation}
The sequence (\ref{eq:4.33}) induces a morphism $q_{2}: B\rightarrow \mathbb{P}(H^{1}(L_{3}^{-1}\otimes L_{2}))$ such that $q_{2}^{*}\mathcal{O}_{\mathbb{P}(H^{1}(L_{3}^{-1}\otimes L_{2}))}(1)$ is of degree $1-\mu_{2}-m=(\mu_{1}-\mu_{2})+\frac{1}{2}$. $\Delta(E)=6$ and (\ref{eq:4.3}) imply $\text{deg}E_{1}=0$ and $\mu_{1}-\mu_{2}=\frac{3}{2}$, which also imply
$\text{deg}L_{1}=\text{deg}L_{2}=0$ and $\text{deg}L_{3}=1$. Thus $q_{2}^{*}\mathcal{O}_{\mathbb{P}(H^{1}(L_{3}^{-1}\otimes L_{2}))}(1)$ is of degree $1-\mu_{2}-m=(\mu_{1}-\mu_{2})+\frac{1}{2}=2$. The sequence (\ref{eq:4.32}) induces a morphism $\psi: B\rightarrow \mathcal{P}$ such that $\psi^{*}\mathcal{O}_{\mathcal{P}}(1)$ is of degree $-\mu_{2}-m=1$.

In Proposition \ref{prop:4.6} when $E$ is semi-stable of degree 0 at generic fiber of $f: X\rightarrow C$, there exists only one point $t_{0}\in C$ such that $E_{t_{0}}$ is not semi-stable and the minimal slop quotient bundle $G$ of rank 2 and $\text{deg}G=-1$.
\end{proof}

\begin{Remark}
If $\phi: B\rightarrow M$ satisfy condition (1), it is a elliptic curve of split type with minimal degree. If $\phi: B\rightarrow M$ satisfy condition (2), it is a elliptic curve in Proposition \ref{prop:3.1}. which implies an elliptic curve of degree 6 may not be an elliptic curve of split type.
\end{Remark}

\begin{Theorem}\label{th:4.8}
When $g> 12$ and $C$ is generic, any essential elliptic curve $\phi: B\rightarrow M=SU_{C}(3, \mathcal{L})$ that passes through the generic points must have $\text{deg}\phi^{*}(-K_{M})\geq 18$.
\end{Theorem}

\end{document}